\documentclass[a4paper, reqno, 11pt]{amsart}

\usepackage[english]{babel}
\usepackage{amsmath}
\usepackage{mathrsfs}
\usepackage{amssymb}
\usepackage{amsthm}
\usepackage{enumerate}
\usepackage{ifthen}
\usepackage{bbm}
\usepackage{color}
\usepackage{xcolor}
\usepackage{graphicx}
\provideboolean{shownotes} 
\setboolean{shownotes}{true}
\usepackage{hyperref}
\usepackage{soul}
\usepackage{geometry}
\newcommand{\margnote}[1]{
\ifthenelse{\boolean{shownotes}}%
{\marginpar{\raggedright\tiny\texttt{#1}}}%
{}%
}

\newcommand{\hole}[1]{
\ifthenelse{\boolean{shownotes}}%
{\begin{center} \fbox{ \rule {.25cm}{0cm}
\rule[-.1cm]{0cm}{.4cm} \parbox{.85\textwidth}{\begin{center}
\texttt{#1}\end{center}} \rule {.25cm}{0cm}}\end{center}}
{}
}
\newtheorem{theorem}{Theorem}[section]   %Theorem 1.1, Lemma 1.2  

\newtheorem{lemma}[theorem]{Lemma}
\newtheorem{proposition}[theorem]{Proposition}

\theoremstyle{definition}
\newtheorem{definition}[theorem]{Definition}
\newtheorem{remark}[theorem]{Remark}

\allowdisplaybreaks

\numberwithin{equation}{section}


\begin{document}

\title[Navier-Stokes-Korteweg equations]{Global unique solvability of the 1D stochastic Navier-Stokes-Korteweg equations}

\author[L. Pescatore]{L. Pescatore}
\address[L.Pescatore]{DISIM - Dipartimento di Ingegneria e Scienze dell'Informazione e Matematica\\ Universit\`a  degli Studi dell'Aquila \\Via Vetoio \\ 67100 L'Aquila \\ Italy}
\email[]{\href{L.Pescatore@}{lorenzo.pescatore@univaq.it}}
\begin{abstract}
We prove the global well-posedness of the one-dimensional Navier-Stokes-Korteweg equations driven by a stochastic multiplicative noise. The analysis is performed for the general case of capillarity and viscosity coefficients $k(\rho)= \rho^\beta, \,  \beta \in \mathbb{R},\, \mu(\rho)=\rho^\alpha, \, \alpha \ge 0,$ which are not coupled through a BD relation. Global existence and uniqueness of solutions is obtained in the regularity class of strong pathwise solutions, which are strong solutions in PDEs and also in the sense of probability.  We first make use of a multi-layer approximation scheme and a stochastic compactness argument to establish the local well-posedness result for any $\alpha$ and $\beta.$ Then, we apply a BD entropy method which provides control of the vacuum states of the density and allows to perform an extension argument. Global well-posedness is thus obtained in the range where there is no vacuum and the strong coercivity condition $2\alpha-4 \le \beta \le 2 \alpha-1,$ introduced in \cite{Germain LeFloch}, holds.
As a byproduct,  we also cover the deterministic setting $\mathbb{F}=0,$ which to the best of our knowledge is likewise an open problem in the fluid dynamics literature.
\end{abstract}
\maketitle
\section{Introduction}
\noindent
We consider the one-dimensional Navier-Stokes-Korteweg equations driven by a stochastic force.  For $x \in \mathbb{T},$ the one-dimensional flat torus, the dynamics is governed by the following set of  equations
\begin{equation}\label{main system}
\begin{cases}
\partial_t \rho+ \partial_x(\rho u)=0, \\ 
\text{d} (\rho u)+ [\partial_x (\rho u^2) + \partial_x p(\rho)]\text{d}t= [\partial_x \mathcal{S} + \partial_x \mathcal{K}]\text{d}t+ \rho \mathbb{F}(\rho,u) \text{d}W,
\end{cases}
\end{equation}
describing a compressible viscous fluid in which capillarity effects are not negligible and are therefore encoded in the Korteweg tensor $\mathcal{K}.$
System \eqref{main system} is supplemented with initial conditions
\begin{equation}\label{initial conditions}
(\rho(x,t), u(x,t))_{| t=0}=(\rho_0(x), u_0(x)),
\end{equation}
which are random variables ranging in a suitable functional space, having normalized mass.  The unknowns $\rho(x,t)$ and $u(x,t)$ denote the density and the velocity of the fluid respectively,  while the isentropic regime is characterized by the constitutive law for the pressure
\begin{equation*}
p(\rho)=  \rho^\gamma, \quad \gamma \ge 1.
\end{equation*}
The viscosity and capillarity tensors are given by
\begin{equation*}
\mathcal{S}(\partial_x u)= \mu(\rho) \partial_x u, \quad \mathcal{K}(\rho, \partial_x \rho)= \rho \partial_x (k(\rho) \partial_x \rho) -\dfrac{1}{2} \big( k(\rho)+ \rho k'(\rho) \big) | \partial_x \rho |^2,
\end{equation*}
for general viscosity and capillarity coefficients 
\begin{equation}\label{power laws cap visc}
\mu(\rho)=  \rho^\alpha, \quad \alpha \ge 0, \quad \; k(\rho)=  \rho^\beta, \quad \beta \in \mathbb{R},
\end{equation}
which do not satisfy any Bresch-Desjardins relation, see for instance the typical assumptions required in \cite{Burtea1},\cite{Burtea2},\cite{Charve} for one-dimensional fluids and \cite{Bresch0} for the multi-dimensional case.
\\
\\
The contribution of the capillarity term in the momentum equation can be written as 
\begin{equation}\label{partial x K}
\partial_x \mathcal{K}(\rho, \partial_x \rho)= \rho \partial_x \bigg( \partial_x(k(\rho) \partial_x \rho)- \dfrac{k'(\rho)}{2} | \partial_x \rho |^2 \bigg)
\end{equation}
and depending on the values of the capillarity exponent $\beta$ several models are considered. For example,  we recall the case $\beta=-1$ which in the multi-dimensional case reads 
\begin{equation*}
\operatorname{div}\mathcal{K}(\rho, \nabla \rho)=2 \rho \nabla \bigg( \dfrac{\Delta \sqrt{\rho}}{\sqrt{\rho}} \bigg),
\end{equation*}
arising in the Quantum hydrodynamics literature \cite{Landau} and the choice $\beta=0,$ corresponding to 
\begin{equation*}
\operatorname{div} \mathcal{K}(\rho, \nabla \rho)= \rho \nabla \Delta \rho,
\end{equation*}
for which \eqref{main system} is also known as the Navier-Stokes-Korteweg system.  A wide literature has been developed in the recent years concerning Korteweg fluids. In particular, the tensor $\mathcal{K}$ has been introduced in \cite{Korteweg} and its structure was  exploited by Dunn and Serrin \cite{Dunn and Serrin}.  Note that the expression \eqref{partial x K} appears also in the description of diffuse interface and phase transition models \cite{Denzler} and a kinetic derivation is discussed in \cite{Gorban}. A detailed overview on the Korteweg tensor $\mathcal{K}$ can be found in \cite{Pescatore}.
Concerning the stochastic force $\rho \mathbb{F}(\rho,u) \text{d}W,$ it is a stochastic multiplicative noise defined in terms of a cylindrical Wiener process $W,$ we refer the reader to Section \ref{Sec2} for the precise definition. The noise is presented in the following form by analogy with the deterministic case where the force usually vanishes on the vacuum set $\{ \rho=0 \}.$ Our analysis can be easily augmented with the addition of a deterministic force of the form $\rho f$ having similar assumptions.
\\
\\
We associate with system \eqref{main system} the following energy functional 
\begin{equation}\label{energy functional}
\mathcal{H}(t)= \int_{\mathbb{T}} \bigg(\frac{1}{2} \rho u^2 +F(\rho)+ k(\rho) \frac{ | \partial_x \rho |^2}{2} \bigg) dx,
\end{equation}
for
\begin{equation*}
F(\rho)= \begin{cases} \dfrac{\rho^\gamma}{\gamma-1}, \quad \, \gamma > 1, \\ \rho \log \rho, \quad \gamma=1,
\end{cases}
\end{equation*}
and we also introduce the BD entropy functional 
\begin{equation}
\mathcal{E}(t)= \int_{\mathbb{T}} \bigg( \frac{1}{2} \rho V^2 +F(\rho)+ k(\rho) \dfrac{ | \partial_x \rho |^2}{2} \bigg)dx,
\end{equation}
which is written in terms of the effective velocity \cite{BD}, \cite{Bresch2}, \cite{Const}, \cite{Const2}, \cite{Mellet}
\begin{equation*}
V=u+ \mu(\rho) \dfrac{\partial_x \rho}{\rho^2}.
\end{equation*}
In the deterministic and unforced system,  the energy functional \eqref{energy functional} formally satisfies the standard equality
\begin{equation}\label{formal energy ineq}
\dfrac{d}{dt}\mathcal{H}(\rho,u) + \int_{\mathbb{T}} \mu(\rho) | \partial_x u |^2 dx=0,
\end{equation}
while the derivation of an entropy balance for $\mathcal{E}(t)$ is non-trivial and some restrictions on the parameters $\alpha$ and $\beta$ occur. To this purpose, Germain and LeFloch \cite{Germain LeFloch} introduced the so-called \textit{strong coercivity condition}. For general capillarity and viscosity coefficients satisfying the power laws \eqref{power laws cap visc}, this condition reduces to 
\begin{equation}\label{SCC} \tag{SCC}
2 \alpha -4 \le \beta \le 2\alpha-1.
\end{equation}
To be precise, \eqref{SCC} characterizes the existence of a non-negative entropy dissipation functional $D[\rho]$ (see Proposition \ref{BD prop} for the precise definition), which again for $\mathbb{F}=0$ formally satisfies
\begin{equation}\label{formal BD ineq}
\dfrac{d}{dt}\mathcal{E}(\rho,u) + \mathcal{D}[\rho]=0.
\end{equation}
Several quantities are controlled by \eqref{formal energy ineq} and  \eqref{formal BD ineq}. In particular, for
\begin{equation}\label{NV} \tag{NV}
0 \le \alpha \le \frac{1}{2}, \quad \text{or} \quad \beta \le -2,
\end{equation}
the $H^1(\mathbb{T})$ norm of $\frac{1}{\rho}$ is controlled in terms of the initial conditions and hence no vacuum states arise in finite time, provided that $\rho_0$ is bounded away from zero.  
\\
The main results of the present paper are a local and a global well-posedness result for strong pathwise solutions, namely strong solutions in both PDEs and Probability sense.  
Specifically, this class of solutions satisfies pointwise in $x$ the time integrated version of \eqref{main system} and the underlying stochastic setting $(\Omega, \mathcal{F}, \mathbb{P})$ is fixed a priori. We remark that the local result holds for any values $\alpha \ge 0$ and $\beta \in \mathbb{R},$ while in order to infer the global regularity theorem we restrict our analysis to the range where \eqref{SCC} and \eqref{NV} hold. We also highlight that uniqueness of global solutions for stochastically forced systems,  gives rise to a wide class of interesting problems concerning the long-time behaviour of solutions. In particular,  the Markovian framework can be defined in a standard manner and existence and uniqueness of invariant measures can be investigated, although the system is out of equilibrium.
Nevertheless, we mention that our results include the cases of different stochastic forces as additive noises and the deterministic case $\mathbb{F}=0$ is covered.  Note also that, by virtue of this global regularity result, the latter model can be used as a suitable approximation of other different fluid dynamical models arising in the literature. This will be addressed in future research.
Lastly, we highlight that our results account for random initial data, which enhances the relevance of the analysis for applications.
\\
\\
Despite the local well-posedness of fluid dynamical models is nowadays a standard result in the literature, several issues arise because of the capillarity tensor $\mathcal{K}$ and due to the stochastic setting. First we adapt deterministic techniques by truncating the nonlinearities of system \eqref{main system} and considering also an appropriate Galerkin approximation. Then, we solve each layer of the approximating scheme and we perform a stochastic compactness argument to deduce convergence results.  We highlight that the nonlinear structure of $\mathcal{K}$ makes it essential to account for density contributions in the truncation argument, in contrast with the situation for the compressible Navier–Stokes equations \cite{Breit Feir Hof 3}, \cite{Feir}. Furthermore,  it also poses non-trivial problems in the proof of the tightness and in the stochastic compactness argument which is based on the appropriate use of Skorokhod's representation theorem. Note however that the following procedure does not keep track of the original probability space. To this end, we prove a pathwise uniqueness result and we use the Gyongy-Krylov method to return to the original probability space and hence recovering strong solutions in probability.
Therefore, local strong pathwise solutions exist and are unique up to a maximal stopping time $\tau$ depending on both $\| \rho \|_{W_x^{2, \infty}}$ and $\| u \|_{W_x^{2, \infty}}.$ 
\\
\\
The extension to global-in-time solutions strongly relies on the a priori estimates and on the BD entropy method used in Section \ref{Sec2}. 
On the one hand, the derivation of such a priori estimates in the case where the capillarity and viscosity coefficients are not coupled through a BD relation is non-trivial. On the other hand, as long as the strong coercivity condition \eqref{SCC} holds, the momenta of the BD entropy $\mathcal{E}(t)$ are controlled in terms of the initial conditions. This, together with the no vacuum condition given by Proposition \ref{prop:vacuum} and \eqref{NV}, allows us to derive an $H^2(\mathbb{T}) \times H^1(\mathbb{T})$ estimate by means of a stochastic Gronwall Lemma. We iterate this procedure to deduce an $H^{s+1}(\mathbb{T}) \times H^s(\mathbb{T})$ estimate for $s > \frac{1}{2}+2$ which controls the maximal stopping time $\tau.$ Finally, the continuity argument is then based on the application of the Markov inequality together with bounds on high order Sobolev norms. 
\\
\\
The presence of the multiplicative noise in the momentum equation complicates the derivation of pathwise estimates. We indeed consider the general form $\mathbb{F}(\rho, u)dW$ and thus there is no effective change of variables which would allow to reduce \eqref{main system} to a suitable PDE with random coefficients. Nevertheless, we deduce estimates on the expectation and momenta of suitable norms, therefore additional  difficulties arise in the extension argument and in the derivation of suitable a priori estimates.
As already pointed out by the structure of the energy functional \eqref{energy functional}, the presence of the Korteweg tensor determines a mismatch in the regularity class of the density and the velocity.  It also poses non-trivial issues concerning loss of derivatives. To be precise, the momentum equation involves three derivatives on the density and the continuity equation yields loss of  derivatives with respect to the velocity $u.$ 
This is highlighted in the derivation of the a priori estimates and also in the proof of Theorem \ref{Main Theorem local}.  To overcome this problem, we introduce the new variable $A(\rho)= \sqrt{\frac{k(\rho)}{\rho}} \partial_x \rho,$ which determines a skew-adjoint structure of \eqref{main system} providing several fundamental cancellations between high-order derivative terms.
By virtue of these considerations, the techniques developed in the theory of compressible fluids cannot be directly applied when the dynamics is influenced by capillarity terms.
\\
\\
We conclude this part with a detailed presentation of the state of the art for Korteweg fluids and related models in the deterministic and also stochastic settings.
The analysis of stochastic compressible flows is very recent.  Most of the literature concerns the compressible Navier-Stokes equations in the case of constant viscosity. In particular,  we mention the monograph by Breit, Feireisl and Hofmanov\'{a} \cite{Feir} and references therein \cite{Breit Hofmanova}, \cite{Breit Feir Hof 1},  \cite{Breit Feir Hof 2}, \cite{Breit Feir Hof 3}, \cite{Breit Feir Hof 4}, \cite{Breit Feir Hof 5}.  See also \cite{Coti} for existence of invariant measures in the one-dimensional case and \cite{Zatorska} for  stability results of weak martingale solutions with linear viscosity. A recent analysis is also provided in the case of transport noise \cite{Breit Feir Hof Zat}, \cite{Breit Feir Hof Much}.
On the other hand, only few results concerning stochastic Korteweg fluids are available. Specifically, well-posedness problems have been addressed in \cite{D.P.S.} and \cite{D.P.S.2}, but only the one-dimensional case with $\beta=-1$ is considered.
\\
\\
In the deterministic setting, a wide literature has been developed in the recent years for viscous capillary fluids. Specifically, the multi-dimensional case mainly focuses on the Quantum-Navier-Stokes equations with linear viscosity, where several results concerning the analysis of weak solutions are available \cite{Spirito2}, \cite{Lacroix},  \cite{Jung qns}, \cite{Vasseur Yu 1}, \cite{Ant6}.
The case of constant capillarity, $\beta=0,$ has been addressed in \cite{Hatt} and \cite{Spirito} where local well-posedness for strong solutions and global existence of weak solutions are obtained. See also \cite{Fanelli1} and \cite{Fanelli2} where the case of rotating fluids with capillarity is considered.
\\
In the one dimensional case, the analysis of well-posedness problems in the case of general capillarity exponent $\beta \in \mathbb{R}$ has been addressed in \cite{Chen} for the case of large initial data, while several results are available provided that $\alpha$ and $\beta$ satisfy some BD relations which usually reduce to $\beta=2\alpha-3,$ see \cite{Burtea1}, \cite{Burtea2}, \cite{Charve}.
On the other hand Germain and LeFloch in \cite{Germain LeFloch} introduced several ranges in which global solutions are expected.  In particular, when the vacuum is not controlled, they proved global existence of finite energy weak solutions in the range $\alpha > \frac{2}{3}$ and $2\alpha-3 \le \beta \le -1.$ Recently the above upper bound on $\beta$ has been improved by Bresch and collaborators in \cite{Ant}. To be precise they cover the case $\alpha > \frac{1}{2}, \, \beta >-2$ and $2\alpha-3 \le \beta \le 2\alpha-1$ corresponding to the \textit{tame capillarity condition} introduced in \cite{Germain LeFloch} for which the capillarity coefficient is dominated by the square of the viscosity and therefore viscous effects are dominant in the dynamics.
Finally, we also highlight that the already mentioned results \cite{D.P.S.}, \cite{D.P.S.2} for the stochastic Quantum-Navier-Stokes equations provide new insights also in the deterministic case.  Specifically, existence of global weak solutions and global well-posedness for strong solutions are proved for $\frac{1}{2} < \alpha \le 1$ and $0 \le \alpha \le \frac{1}{2}$ respectively.
\\
\\
Among the inviscid models for capillary fluids we highlight the Quantum-Hydrodynamics system which has been widely studied in the series of papers \cite{Ant1}, \cite{Ant2}, \cite{Ant3}, \cite{Ant4} in the framework of global finite energy weak solutions.
Several problems remain open for the Euler-Korteweg system. In particular, no results of global existence of weak solutions are available, even in the one-dimensional case, while the local existence of smooth solutions has been obtained in \cite{Benz}, \cite{Benz2}. We refer to \cite{Aud} for small data well-posedness.  
For completeness, we also mention the non-uniqueness result based on convex integration techniques obtained in \cite{Feir. Don.} in case $\beta=0,$ and the relative entropy approaches provided in \cite{Bresch0} and \cite{Giess}.
To conclude, singular limits have been investigated for viscous and also inviscid models, here we collect a non-exhaustive list \cite{Ant Cac}, \cite{Ant5}, \cite{Donatelli}, \cite{Ant0}, \cite{Cag-Don-Hien}.
\\
\\
\textbf{Outline of the paper.} The remaining part of the paper is organized as follows.  In Section \ref{Sec2} we give the definitions of strong pathwise solutions and we state the main results we aim to prove. Section \ref{Sec3} is devoted to the derivation of the a priori estimates. In particular we derive the BD entropy inequality and we also provide suitable upper and lower bounds for the density, together with some higher order derivative estimates. In Section \ref{Sec4} we prove the local well-posedness result Theorem \ref{Main Theorem local}. Finally, Section \ref{Sec5} contains the proof of Theorem \ref{Main Theorem global},  which is the main result of the paper.
\\
\\
\textbf{Acknowledgements.} The author gratefully acknowledges the partial support by the Gruppo Na\-zio\-na\-le per l’Analisi Matematica, la Probabilit\`a e le loro
Applicazioni (GNAMPA) of the Istituto Nazionale di Alta Matematica
(INdAM), and by the PRIN 2020 ``Nonlinear evolution PDEs, fluid
dynamics and transport equations: theoretical foundations and
applications'' and by the PRIN2022
``Classical equations of compressible fluids mechanics: existence and
properties of non-classical solutions''. 
\section{Definitions and main results}\label{Sec2}
\noindent
We start our discussion with a brief introduction to the stochastic element of our analysis. Then we focus on the description of strong pathwise solutions and to the presentation of the main results achieved in this paper.
\subsection{Assumptions on the stochastic elements}
We denote by $(\Omega, \mathcal{F},(\mathcal{F}_t)_{t \ge 0},\mathbb{P})$  a stochastic basis with right-continuous filtration. Then, we define the cylindrical $(\mathcal{F}_t)$ Wiener process as 
\begin{equation} \label{Cyl W}
W(t)=\sum_{k=1}^\infty e_kW_k(t), \quad t\in [0,T]
\end{equation}
in terms of a sequence of mutually independent real-valued Brownian motions $ (W_k)_{k\in \mathbb{N}}$ and an orthonormal basis $(e_k)_{k\in \mathbb{N}}$ of a separable Hilbert space $\mathfrak{U}.$
The stochastic integral is then understood in the  It$\hat{\text{o}}$ sense,   see \cite{Da Prato} and also \cite{Feir} for applications to compressible fluid dynamical equations. The driving process $\mathbb{F}(\rho,u): \mathfrak{U} \rightarrow L^2(\mathbb{T})$ is similarly defined by superposition as 
\begin{equation*}
\mathbb{F}(\rho,u)e_k=F_k(\cdot,\rho(\cdot),u(\cdot)),
\end{equation*}
with assigned regularity on the coefficients $$F_k: \mathbb{T} \times [0, \infty) \times \mathbb{R} \rightarrow \mathbb{R}.$$ In particular, we assume the following growth assumptions uniformly in $x \in \mathbb{T}$ 
\begin{equation}\label{F1}
F_k(\cdot,0,0)=0,
\end{equation}
\begin{equation}\label{F w}
| F_k(x,\rho,q) | \le f_k (1+|u|),
\end{equation}
\begin{equation}\label{F2}
| \partial^l_{x,\rho,q} F_k(x, \rho,q) | \le f_k,  \quad \sum_{k=1}^\infty f_k < \infty \quad  \text{for all} \; l\in {1,...,s}, \quad \{ f_k \}_{k\in \mathbb{N}} \subset \mathbb{R}.
\end{equation}
Namely, $F_k$ are $C^s$-functions, for $s \in \mathbb{N}$ complying with the space regularity of solutions.
The stochastic integral $$ \int_{0}^{t} \rho \mathbb{F}(\rho, u) \text{d}W= \sum_{k=1}^{\infty} \int_{0}^{t} \rho F_k( \cdot, \rho, u) \text{d}W_k$$
is then understood in the It\^{o} sense. We also highlight that the sum in \eqref{Cyl W} is not a priori expected to be converging in $\mathfrak{U}$. This issue is typically addressed by introducing the space $ \mathfrak{U}_0 \supset \mathfrak{U}$ defined by
\begin{equation*}
\mathfrak{U}_0= \bigg\{ v= \sum_{k \ge 1}^{} c_k e_k; \quad \sum_{k \ge 1}^{} \dfrac{c^2_k}{k^2} < \infty \bigg\},
\end{equation*}
which is endowed with the following norm
\begin{equation*}
\| v \|^2_{\mathfrak{U}_0}= \sum_{k \ge 1}^{} \dfrac{c^2_k}{k^2}, \quad v= \sum_{k \ge 1}^{} c_k e_k.
\end{equation*}
Nevertheless, the process $W$ has $\mathbb{P}$-a.s. $C([0,T];\mathfrak{U}_0)$ sample paths and the superposition operator belongs to the Hilbert-Schmidt class, see Section 2.3 in \cite{Feir}.
\\
\\
We also account randomness in the initial data. In particular, $(\rho_0,u_0)$ are $\mathcal{F}_0$-measurable random variables satisfying 
\begin{equation} \label{C.I STRONG}
 C^{-1} \le \rho_0 \le C,  \quad \rho_0 \in H^{s+1}(\mathbb{T}),  \quad u_0 \in H^s(\mathbb{T}) \quad \mathbb{P}\text{-a.s.}%, \quad \partial_x \bigg( \dfrac{1}{\sqrt{\rho_0}}\bigg) \in L^2(\mathbb{T})% 
\end{equation}
for a deterministic constant $C > 0$ and $s>7/2.$ 
Clearly, to deduce the estimates on the expectation and momenta of the energy-entropy functional, together with the high-order Sobolev norms considered in Section 3, it is necessary to assume that the same regularity holds at the initial time. Thus, for the global regularity of solutions, the following additional conditions are required
\begin{equation} \label{C.I. Momenta strong}
\rho_0 \in L^p(\Omega; H^{s+1}(\mathbb{T})),  \quad u_0 \in L^p(\Omega; H^s(\mathbb{T})), \quad \text{for all} \;p  \in [1, \infty).
\end{equation}
\subsection{Definitions and main theorems}
In the present Section, we fix $(\Omega, \mathcal{F},(\mathcal{F}_t)_{t \ge 0},\mathbb{P})$ and $W$ to be a stochastic basis with right continuous filtration and an $(\mathcal{F}_t)$-cylindrical Wiener process respectively.  The stochastic setting is then assigned and we introduce the definition of strong pathwise solution.
\begin{definition} \label{Def1}
Let $(\rho_0,u_0)$ be an $\mathcal{F}_0-$measurable random variable with regularity \eqref{C.I STRONG} for some $s > \frac{5}{2}.$
We say that a triplet $(\rho,u,\tau)$ is a \textit{local strong pathwise solution} to system \eqref{main system}-\eqref{initial conditions}, provided
\begin{itemize}
\item[(1)]
 $\tau$ is a $\mathbb{P}-a.s.$ strictly positive $(\mathcal{F}_t)$-stopping time;
 \item[(2)]
 the density $\rho$ is a $H^{s+1}(\mathbb{T})$-valued $(\mathcal{F}_t)$-progressively measurable stochastic process satisfying
 $$ \rho(\cdot \land \tau) >0, \quad \rho(\cdot \land \tau) \in C([0,T];H^{s+1}(\mathbb{T})) \quad \mathbb{P}-a.s.; $$
 \item[(3)]
 the velocity $u$ is a $H^s(\mathbb{T})$-valued $(\mathcal{F}_t)$-progressively measurable stochastic process such that 
$$ u(\cdot \land \tau) \in C([0,T];H^s(\mathbb{T})) \cap L^2([0,T];H^{s+1}(\mathbb{T})) \quad \mathbb{P}-a.s.; $$
\item[(4)]
the continuity equation $$ \rho(t \land \tau)= \rho_0- \int_{0}^{t\land \tau} \partial_x(\rho u) ds $$
holds for all $t \in [0,T] \ \mathbb{P}-a.s.;$
\item[(5)] the momentum equation 
\begin{equation*}
\begin{split}
(\rho u)(t \land \tau)&= \rho_0 u_0 -\int_{0}^{t \land \tau} \partial_x(\rho u^2+p(\rho))ds +\int_{0}^{t \land \tau} \partial_x(\mu(\rho) \partial_x u)ds \\ &+\int_{0}^{t \land \tau} \partial_x \mathcal{K}ds + \int_{0}^{t \land \tau} \rho \mathbb{F}(\rho, u)dW
\end{split}
\end{equation*}
holds for all $t \in [0,T] \ \mathbb{P}-a.s.;$
\end{itemize}
\end{definition}
\noindent
The definitions of maximal and global strong pathwise solutions are then stated as follows
\begin{definition}\label{Def2}
Let $(\rho_0,u_0)$ be an $\mathcal{F}_0-$measurable random variable with regularity \eqref{C.I STRONG}.
A quadruplet $(\rho,u,(\tau_R)_{R \in \mathbb{N}},\tau)$ is called a \textit{maximal strong pathwise solution} to system \eqref{main system}-\eqref{initial conditions}, provided that
\begin{enumerate}
\item[(1)]
 $\tau$ is a $\mathbb{P}-a.s.$ strictly positive $(\mathcal{F}_t)$-stopping time;
 \item[(2)] 
 $\{ \tau_R \}_{R \in \mathbb{N}}$ is an increasing sequence of $(\mathcal{F}_t)$-stopping times such that
\begin{itemize}
\item 
$\tau_R < \tau$ on the set $[\tau < \infty],$
\item 
$\lim_{R \rightarrow \infty} \tau_R=\tau $ a.s. 
\item The following condition holds
\begin{equation}\label{limit cond}
\sup_{t \in [0,\tau_R]} \big( \| \rho (t) \|_{W_x^{2,\infty}}+  \| u(t) \|_{W_x^{2,\infty}} \big) \ge R \ \text{in} \ [\tau < \infty];
\end{equation}
\end{itemize}
\item[(3)]
For any $R \in \mathbb{N},$ each triplet $(\rho,u,\tau_R)$ is a local strong pathwise solution in the sense of Definition \ref{Def1}.
\end{enumerate}
In the case of $(\rho,u,(\tau_R)_{R \in \mathbb{N}},\tau)$ being a maximal strong pathwise solution with $\tau=\infty, \;  a.s. ,$ then we say that such a solution is global. 
\end{definition}
\noindent
Formally,  as $R \rightarrow \infty$ \eqref{limit cond} reduces to 
\begin{equation*}
\phi(\tau)=\sup_{t \in [0,\tau]} \big( \| \rho (t) \|_{W_x^{2,\infty}}+  \| u(t) \|_{W_x^{2,\infty}} \big) = \infty \quad  \text{in} \ [\tau < \infty].
\end{equation*}
On the other hand, in the case of global maximal strong pathwise solutions,  the quantity $\phi(\infty)$ may be finite or infinite.  This dichotomy is a quite standard scenario in the case of stochastically forced PDE, see for instance \cite{Deb} and \cite{Glatt-Holtz} for a similar discussion for global solutions of incompressible fluid equations and to \cite{D.P.S.} for closely related models of compressible fluids.
\\
\\
The main results of this paper are the following two theorems.  Theorem \ref{Main Theorem local} establishes existence and uniqueness of a local in time maximal strong pathwise solution for any $\alpha$ and $\beta.$ Theorem \ref{Main Theorem global} relies on a continuity argument which allows to extend the unique local solution to a global one. Global well-posedness is therefore provided in the range given by the strong coercivity condition \eqref{SCC} introduced in \cite{Germain LeFloch} and no vacuum \eqref{NV}.
\begin{theorem}\label{Main Theorem local}
Let $s \in \mathbb{N}$ satisfy $ s > \frac{7}{2}.$ Assume the coefficients $F_k$ satisfy \eqref{F1}-\eqref{F2}. Then for any $\alpha \ge 0, \; \beta \in \mathbb{R}, $ there exists a unique maximal local strong pathwise solution $(\rho,u,(\tau_R)_{R \in \mathbb{N}},\tau)$ to system \eqref{main system}-\eqref{initial conditions} in the sense of Definition \ref{Def2}, provided that the initial conditions satisfy \eqref{C.I STRONG}.
\end{theorem}
\begin{theorem}\label{Main Theorem global} 
Under the same hypothesis of Theorem \ref{Main Theorem local},  assume in addition that \eqref{SCC} and \eqref{NV} hold and that the initial conditions satisfy \eqref{C.I. Momenta strong}. Then, there exists a unique maximal global strong pathwise solution $(\rho,u,(\tau_R)_{R \in \mathbb{N}},\tau)$ to system \eqref{main system}-\eqref{initial conditions} in the sense of Definition \ref{Def2}.
\end{theorem}
\begin{remark}
Some remarks are in order
\begin{itemize}
\item [1)]
By virtue of a standard mollification argument and stochastic compactness, we can prove the global well-posedness result given by Theorem \ref{Main Theorem global} for strong pathwise having space regularity $H^2_x \times H^1_x,$ provided that the same regularity is assumed on the initial data and $\rho_0$ is bounded away from the vacuum. Also the regularity on the stochastic force can be relaxed to $F_k \in C_x^2.$
\\
\item [2)]
The underlying probability space is not endowed with a topological structure. Therefore, it is fundamental to apply a pathwise uniqueness argument to recover strong solutions in probability. This procedure remains valid provided that $s>\frac{7}{2}.$ In contrast, such a condition is not required in a purely deterministic setting, where our argument for the proof of Theorem \ref{Main Theorem local} would apply for a lower differentiability exponent $s> \frac{5}{2}.$
\\
\item [3)]
The BD entropy estimate Proposition \ref{BD prop} and therefore Theorem \ref{Main Theorem global} can be easily extended to the case $x \in \mathbb{R},$ with the asymptotic conditions given in \cite{Germain LeFloch}. The case of domains with boundary is still an open problem.
\\
\item [4)]
Since $\rho$ and $u$ have continuous trajectories in $H^{s+1}$ and $H^{s}  $ then a blow up of the $H^s$ norm of the solution coincides with a blow up in the $W^{2,\infty}$ norm at time $\tau.$ Also, the notion of uniqueness refers only to the solution $(\rho,u)$ and the maximal stopping time $\tau.$ 
\end{itemize}
\end{remark}
\noindent
The proof of Theorem \ref{Main Theorem local} builds on deterministic and stochastic techniques. Specifically we make use of a multi-layer approximation scheme arising in deterministic settings, together with stochastic compactness arguments. We tailored the theory developed in the monograph \cite{Feir} and \cite{Breit Feir Hof 3} for the Navier-Stokes equation to the case of capillary fluids, in which the nonlinearity due to the Korteweg tensor plays a crucial role. 
Thereafter, the BD‑entropy method provides the fundamental tools used in the extension argument. More precisely, Proposition \ref{BD prop} yields an appropriate time‑integrability condition on the second derivative of the density, which together with the no‑vacuum result established in Proposition \ref{prop:vacuum}, allows us to control the blow-up norms through high‑order Sobolev estimates
\section{A priori estimates}\label{Sec3}
\noindent
This Section contains all the a priori estimates needed in our analysis. In particular, the main elements are the BD entropy method considered in Proposition \ref{BD prop} and the higher order a priori estimates Proposition \ref{Prop h2 est} and Proposition \ref{Prop global s+1 s},  which can be obtained when the vacuum is controlled.  In what follows, $T$ denotes an arbitrarily large time and $\tau$ represent a generic stopping which formally encodes the maximal life span of strong solutions. We start by proving the following standard energy estimate
\begin{proposition}\label{energy prop}
Let $(\rho,u)$ be a strong solution of \eqref{main system}-\eqref{initial conditions}, then the following energy estimate holds
\begin{equation}\label{energy visk}
\begin{split}
& \mathbb{E} \bigg[ \sup_{t \in [0,T \land \tau]}\mathcal{H}(t) \bigg]^p +\mathbb{E} \bigg[ \int_{0}^{T\land \tau} \int_{\mathbb{T}} \mu(\rho)| \partial_x u |^2 dxdt \bigg]^p  \lesssim 1+ \mathbb{E}\bigg[ \mathcal{H}(0) \bigg]^p,
\end{split}
\end{equation}
for any $p \in [1, \infty).$
\end{proposition}
\begin{proof}
We apply the It\^{o} formula, see also Theorem A.4.1 \cite{Feir}, to the functional 
\begin{equation*}
F(\rho,u)(t) =\int_{\mathbb{T}} \frac{1}{2}\rho(t)|u(t)|^2dx,
\end{equation*}
therefore we obtain the following identity
\begin{equation} \label{En F}
\begin{split}
& \int_{\mathbb{T}} \frac{1}{2} \rho |u|^2(t)dx=\int_{\mathbb{T}} \frac{1}{2} \rho_0 |u_0|^2dx-\int_{0}^{t} \int_{\mathbb{T}} [ \rho u^2 \partial_x u +u\partial_x p(\rho)]dxds \\ & +\int_{0}^{t} \int_{\mathbb{T}} [u\partial_x(\mu(\rho)\partial_x u)+  u \partial_x \mathcal{K}] dxds+\frac{1}{2}\int_{0}^{t} \int_{\mathbb{T}} \sum_{k\in \mathbb{N}} \rho | F_k(\rho,u)|^2dxdt \\ & -\int_{0}^{t} \int_{\mathbb{T}} \frac{1}{2} |u|^2 \partial_x(\rho u )dxds+ \int_{0}^{t} \int_{\mathbb{T}} \sum_{k \in \mathbb{N}} \rho u \mathbb{F}(\rho,u)dxdW.
\end{split}
\end{equation}
Then we note that integrating by parts multiple times,  several cancellations occur and 
\begin{equation*}
\int_{\mathbb{T}} \partial_x p(\rho) u dx= \dfrac{d}{dt} \int_{\mathbb{T}} F(\rho) dx,
\end{equation*}
while the Korteweg tensor $\mathcal{K}$ give rise to the following term
\begin{equation*}
\begin{split}
 \int_{\mathbb{T}} \partial_x \mathcal{K} u \, dx & = \int_{\mathbb{T}} \partial_x \bigg( \partial_x(k(\rho) \partial_x \rho)- \dfrac{k'(\rho)}{2} | \partial_x \rho |^2 \bigg) \rho u \, dx = - \dfrac{d}{dt} \int_{\mathbb{T}} k(\rho)\dfrac{ | \partial_x \rho |^2}{2}dx.
\end{split}
\end{equation*}
Therefore the energy balance \eqref{En F}  can be written as 
\begin{equation}\label{energy balance stoch}
\begin{split}
& \int_{\mathbb{T}} \dfrac{1}{2} \rho |u|^2+ \dfrac{\rho^\gamma}{\gamma-1}+ k(\rho)\dfrac{ | \partial_x \rho |^2}{2}dx + \int_{0}^{t} \int_{\mathbb{T}} \mu(\rho)| \partial_x u |^2 dxds \\ & = \int_{\mathbb{T}} \dfrac{1}{2} \rho_0 |u_0|^2+ \dfrac{\rho_0^\gamma}{\gamma-1}+ k(\rho_0)\dfrac{ | \partial_x \rho_0 |^2}{2}dx+ \frac{1}{2} \sum_{k=1}^{\infty}\int_{0}^{t}\int_{\mathbb{T}} \rho |F_k|^2dxds+\int_{0}^{t} \int_{\mathbb{T}} \rho u \mathbb{F}(\rho,u) dxdW.
\end{split}
\end{equation}
In order to estimate the stochastic integral in the right-hand-side of \eqref{energy balance stoch} we make use of the Burkholder-Davis-Gundy inequality,  Proposition 2.3.8 in \cite{Feir}. To this purpose, we take the $\sup$ in time and the $p$-th power in \eqref{energy balance stoch}, then we  apply the expectation to get
\begin{equation}\label{energy stoch}
\begin{split}
& \mathbb{E} \bigg| \sup_{t \in [0.T\land \tau]} \int_{\mathbb{T}} \bigg[ \dfrac{1}{2} \rho |u|^2+ \dfrac{\rho^\gamma}{\gamma-1}+ k(\rho)\dfrac{ | \partial_x \rho |^2}{2} \bigg] dx \bigg|^p + \mathbb{E} \bigg| \int_{0}^{T\land \tau} \int_{\mathbb{T}} \mu(\rho)| \partial_x u |^2 dxdt \bigg|^p \\ &  \lesssim \mathbb{E} \bigg| \int_{\mathbb{T}} \bigg[ \dfrac{1}{2} \rho_0 |u_0|^2+ \dfrac{\rho_0^\gamma}{\gamma-1}+ k(\rho_0)\dfrac{ | \partial_x \rho_0 |^2}{2} \bigg] dx \bigg|^p+ \mathbb{E} \bigg | \frac{1}{2} \sum_{k=1}^{\infty}\int_{0}^{T\land \tau}\int_{\mathbb{T}} \rho |F_k|^2dxdt \bigg|^p \\ & +\mathbb{E} \bigg[ \sup_{t \in [0,T \land \tau]} \bigg| \int_{0}^{t} \int_{\mathbb{T}} \rho u \mathbb{F}(\rho,u) dxdW \bigg|^p \bigg].
\end{split}
\end{equation}
Thus, by using Proposition 2.3.8 in \cite{Feir} and \eqref{F w}  we have 
\begin{equation}\label{stoch int enrgy estimate}
\begin{split}
\mathbb{E} \bigg[ \sup_{t \in [0.T\land \tau]} \bigg|  \int_{0}^{t} \int_{\mathbb{T}} \rho u \mathbb{F}(\rho,u) dxdW \bigg|^p \bigg] & \lesssim \mathbb{E} \bigg[ \int_{0}^{T\land \tau} \sum_{k \in \mathbb{N}} \bigg| \int_{\mathbb{T}} \rho u F_k dx \bigg|^2 dt \bigg]^{\frac{p}{2}} \\ & \le \mathbb{E}\bigg[ \int_{0}^{T\land \tau}  \bigg| \int_{\mathbb{T}} ( \rho+ \rho |u|^2) dx  \bigg|^2 dt \bigg]^{\frac{p}{2}},
\end{split}
\end{equation}
while the It\^{o} correction term can be estimated by \eqref{F1}-\eqref{F2} as
\begin{equation*}
\begin{split}
\sum_{k \in \mathbb{N}} \int_{0}^{T\land \tau} \int_{\mathbb{T}} \rho |F_k|^2dxdt \le \int_{0}^{T\land \tau} \sum_{k \in \mathbb{N}} f^2_k \int_{\mathbb{T}} ( \rho + \rho |u|^2) dxdt \lesssim \int_{0}^{T\land \tau} \int_{\mathbb{T}} ( \rho+ \rho |u|^2) dxdt.
\end{split}
\end{equation*}
To conclude, we apply Young inequality in \eqref{stoch int enrgy estimate} and the standard Gronwall Lemma to deduce
\begin{equation*}
\begin{split}
& \mathbb{E} \bigg| \sup_{t \in [0,T\land \tau]} \int_{\mathbb{T}} \bigg[\dfrac{1}{2}\rho |u|^2+ \dfrac{\rho^\gamma}{\gamma-1}+ k(\rho)\dfrac{ | \partial_x \rho |^2}{2} \bigg] dx \bigg|^p +\mathbb{E} \bigg| \int_{0}^{T\land \tau} \int_{\mathbb{T}} \mu(\rho)| \partial_x u |^2 dxdt \bigg|^p \\ & \lesssim 1+ \mathbb{E}\bigg|  \int_{\mathbb{T}} \bigg[\dfrac{1}{2}\rho_0 |u_0|^2+ \dfrac{\rho_0^\gamma}{\gamma-1}+ k(\rho_0)\dfrac{ | \partial_x \rho_0 |^2}{2} \bigg] dx \bigg|^p 
\end{split}
\end{equation*}
which concludes our proof.
\end{proof}
\noindent
The next proposition is the BD entropy estimates.  To begin with, we introduce the effective velocity 
\begin{equation*}
V:=u+ \mu(\rho) \dfrac{\partial_x \rho}{\rho^2}
\end{equation*}
and we observe that the quantity $Q=\mu(\rho) \dfrac{\partial_x \rho}{\rho^2}$ satisfies the following transport equation
\begin{equation*}
\partial_t Q+ u \partial_x Q= -\dfrac{1}{\rho} \partial_x (\mu(\rho) \partial_x u),
\end{equation*}
thus the system in the new variables $(\rho, \rho V)$ reads 
\begin{equation*}
\begin{cases}
\partial_t \rho+ \partial_x(\rho u)=0, \\ 
\text{d} (\rho V)+ [\partial_x (\rho u V) + \partial_x p(\rho)]\text{d}t= \partial_x \mathcal{K}\text{d}t+ \rho \mathbb{F}(\rho,u) \text{d}W.
\end{cases}
\end{equation*}
Then we infer the following result:
\begin{proposition} \label{BD prop}
Let $(\rho,u)$ be a strong solution of \eqref{main system}-\eqref{initial conditions}, assume that the strong coercivity condition \eqref{SCC} holds. Then for any $p \in [1, \infty),$ the following entropy inequality is satisfied
\begin{equation}\label{BD ENTROPY visk}
\begin{split}
& \mathbb{E} \bigg[ \sup_{t \in [0,T\land \tau]} \mathcal{E}(t) \bigg]^p +\mathbb{E} \bigg[ \int_{0}^{T \land \tau} \mathcal{D}[\rho]ds \bigg]^p \lesssim 1+ \mathbb{E}\bigg[ \mathcal{E}(0) \bigg]^p 
\end{split}
\end{equation}
for an entropy dissipation functional $D[\rho]:=D_{\alpha,\gamma}[\rho]+D_{\alpha,\beta}[\rho]$ with, 
\begin{equation}\label{entropy D pres}
\mathcal{D}_{\alpha,\gamma}[\rho]= \int_{\mathbb{T}} \dfrac{4\gamma}{(\gamma+\alpha-1)^2} | \partial_x \rho^{\frac{\gamma+\alpha-1}{2}} |^2 dx
\end{equation}
and 
\begin{equation}\label{entropy D cap}
\mathcal{D}_{\alpha,\beta}[\rho]= \int_{\mathbb{T}} \bigg( \frac{4}{(\alpha+\beta+1)^2} | \partial_{xx} \rho^{\frac{\alpha+ \beta+1}{2}}|^2 + c(\alpha,\beta) | \partial_x \rho^{\frac{\alpha+\beta+1}{4}} |^4 \bigg) dx,
\end{equation}
with the constant $c(\alpha,\beta)$ defined by 
\begin{equation*}
c(\alpha,\beta)= \dfrac{64}{(\alpha+\beta+1)^2} \bigg[ \dfrac{(\alpha-\beta-1)(1-\alpha)}{(\alpha+\beta+1)^2}-\dfrac{\beta}{3(\alpha+\beta+1)} \bigg] 
\end{equation*}
\end{proposition}
\begin{remark}\label{degenerate case}
The term $\sqrt{\rho}V$ corresponds to $\frac{\partial_x \rho^{(\alpha-\frac{1}{2})}}{(\alpha-\frac{1}{2})}.$ By construction, the quantity $\frac{\partial_x \rho^\theta}{\theta}$ reduces to $\partial_x \log \rho$ in the degenerate case $\theta=0.$
\end{remark}
\begin{proof}
With similar lines of argument with respect to the proof of Proposition \ref{energy prop}, we apply It\^{o} formula to the functional 
\begin{equation*}
F(\rho,V)(t) =\int_{\mathbb{T}} \frac{1}{2}\rho(t)|V(t)|^2dx,
\end{equation*}
and we obtain the following identity
\begin{equation} \label{Ent F}
\begin{split}
& \int_{\mathbb{T}} \frac{1}{2} \rho |V|^2(t)dx=\int_{\mathbb{T}} \frac{1}{2} \rho_0 |V_0|^2dx-\int_{0}^{t} \int_{\mathbb{T}} [ \rho V u \partial_x V +V\partial_x p(\rho)]dxds \\ & +\int_{0}^{t} V \partial_x \mathcal{K} dxds+\frac{1}{2}\int_{0}^{t} \int_{\mathbb{T}} \sum_{k\in \mathbb{N}} \rho | F_k(\rho,u)|^2dxdt \\ & - \int_{0}^{t} \int_{\mathbb{T}} \frac{1}{2} |V|^2 \partial_x(\rho u )dxds+ \int_{0}^{t} \int_{\mathbb{T}} \sum_{k \in \mathbb{N}} \rho V \mathbb{F}(\rho,u)dxdW.
\end{split}
\end{equation}
The contribution of the terms involving only $u$ have been already carried out in the proof of Proposition \ref{energy prop}. Hence we focus on the terms involving the transport part $Q.$ Specifically, we have 
\begin{equation}\label{pressure bd}
\begin{split}
& \int_{\mathbb{T}} \partial_x p(\rho) Q \, dx = \dfrac{4\gamma}{(\gamma+\alpha-1)^2}\int_{\mathbb{T}} | \partial_x \rho^{\frac{\gamma+\alpha-1}{2}} |^2 \, dx,
\end{split}
\end{equation}
while the analysis of the capillary term is more delicate and strongly relies on the strong coercivity condition \eqref{SCC}. To be precise, integrating by parts several times we have
\begin{equation}\label{capillarity BD}
-\int_{\mathbb{T}} \partial_x \mathcal{K} \cdot Q \, dx = \frac{4}{(\alpha+\beta+1)^2} \int_{\mathbb{T}} | \partial_{xx} \rho^{\frac{\alpha+ \beta+1}{2}} |^2 dx  + c(\alpha, \beta) \int_{\mathbb{T}} |\partial_x \rho^{\frac{\alpha+\beta+1}{4}} |^4 dx.
\end{equation}
Note that the constant $c(\alpha,\beta)$ is not necessarily positive. On the other hand,  the strong coercivity condition \eqref{SCC} characterizes the range in which $D_{\alpha,\beta}[\rho]$ is non-negative.  The proof of the above dichotomy is given in Appendix \ref{Deterministic tools}, see Proposition \ref{char scc}. The main element of the proof is the appropriate use the functional inequality given in Proposition \ref{Functional ineq }, which can be also find in \cite{Alazard}, Proposition 2.8.
\\
\\
By virtue of \eqref{pressure bd} and \eqref{capillarity BD} we rewrite the energy balance as 
\begin{equation}\label{entropy stoch}
\begin{split}
& \mathcal{E}(t) + \dfrac{4\gamma}{(\gamma+\alpha-1)^2} \int_{0}^{t}\int_{\mathbb{T}} | \partial_x \rho^{\frac{\gamma+\alpha-1}{2}} |^2 \, dxds +  \frac{4}{(\alpha+\beta+1)^2} \int_{0}^{t} \int_{\mathbb{T}} | \partial_{xx} \rho^{\frac{\alpha+ \beta+1}{2}} |^2 dxds  \\ & + c(\alpha, \beta)\int_{0}^{t} \int_{\mathbb{T}}|\partial_x \rho^{\frac{\alpha+\beta+1}{4}} |^4 dxds = \mathcal{E}(0)+ \frac{1}{2} \sum_{k=1}^{\infty}\int_{0}^{t}\int_{\mathbb{T}} \rho |F_k|^2dxds+ \int_{0}^{t}\int_{\mathbb{T}} \rho V \mathbb{F}(\rho,u) dxdW.
\end{split}
\end{equation}
Concerning the right-hand-side of \eqref{entropy stoch} we have
\begin{equation*}
\begin{split}
\sum_{k \in \mathbb{N}} \int_{0}^{T\land \tau} \int_{\mathbb{T}} \rho |F_k|^2dxdt \le \int_{0}^{T\land \tau} \sum_{k \in \mathbb{N}} f^2_k \int_{\mathbb{T}} ( \rho + \rho |u|^2) dxdt \lesssim \int_{0}^{T\land \tau} \int_{\mathbb{T}} ( \rho+ \rho |u|^2) dxdt,
\end{split}
\end{equation*}
therefore similarly to Proposition \ref{energy prop} we estimate the stochastic integral as follows
\begin{equation}\label{stoch int bd entr}
\begin{split}
\mathbb{E} \bigg[ \sup_{t \in [0.T\land \tau]} \bigg| \int_{0}^{t} \int_{} \rho V \mathbb{F}(\rho,u) dxdW \bigg|^p \bigg] & \lesssim \mathbb{E} \bigg[ \int_{0}^{T\land \tau} \sum_{k \in \mathbb{N}} \bigg| \int_{\mathbb{T}} \rho V F_k dx \bigg|^2 dt \bigg]^{\frac{p}{2}} \\ & \le \mathbb{E}\bigg[ \int_{0}^{T\land \tau}  \bigg| \int_{\mathbb{T}} ( \rho+ \rho |V|^2) dx \bigg|^2 dt \bigg]^{\frac{p}{2}}
\end{split}
\end{equation}
and we use Young inequality and a standard Gronwall argument to deduce \eqref{BD ENTROPY visk}.
\end{proof}
\noindent
The next proposition provides an upper and a lower bound for the density. The latter estimate is ensured by the no vacuum condition \eqref{NV}.
\begin{proposition}\label{prop:vacuum}
Let $(\rho,u)$ be a strong solution of \eqref{main system}-\eqref{initial conditions}, then for all $p\in [1, \infty)$ the following bounds hold:
\begin{itemize}
\item [(1)] For any $\alpha, \beta$
\begin{equation}\label{rho L infty beta}
\rho^{\frac{\beta}{2}+1} \in L^p(\Omega;L^{\infty}(0,T\land \tau;L^\infty(\mathbb{T}))).
\end{equation}
\item [(2)] For any $\alpha, \beta$ satisfying \eqref{SCC}
\begin{equation}\label{rho L infity alpha}
\rho^{\alpha-\frac{1}{2}} \in L^p(\Omega;L^{\infty}(0,T\land \tau;L^\infty(\mathbb{T})))
\end{equation}
and 
\begin{equation}\label{rho L infty}
\rho \in L^p(\Omega;L^{\infty}(0,T\land \tau;L^\infty(\mathbb{T}))).
\end{equation}
\item [(3)]
For any $\alpha, \beta$ such that $\beta \le -2$ or both \eqref{SCC} and \eqref{NV} hold
\begin{equation}\label{1 div rho L infty}
\dfrac{1}{\rho} \in L^p(\Omega;L^{\infty}(0,T\land \tau;L^\infty(\mathbb{T}))).
\end{equation}
\end{itemize}
\end{proposition}
\begin{proof}
By virtue of the mean value theorem, we infer the existence of a point $\bar x$ such that for all $t \in (0,T\land \tau)$ and $\omega \in \Omega,$ $$\rho(\bar x,t,  \omega)= \int_{\mathbb{T}} \rho(x,t,\omega) dx.$$
Then since $\rho$ has unit mass,  we use the fundamental theorem of calculus to deduce
\begin{equation}\label{FTC 0}
\rho^\theta(x,t,\omega)=1+ \int_{\bar x}^{ x} \partial_y \rho^\theta dy \le 1+ \int_{\mathbb{T}} |\partial_x \rho^\theta| dx,
\end{equation}
for any $\theta \in \mathbb{R}.$
Thus by taking in order the $\sup$ in space and time in \eqref{FTC 0}, the $p-$th power and the expectation, we choose $\theta= \frac{\beta}{2}+1$ and use Proposition \ref{energy prop} to deduce
 \eqref{rho L infty beta}.  Similarly we choose $\theta= \alpha-\frac{1}{2}$ and use Proposition \eqref{BD prop} to infer \eqref{rho L infity alpha}, provided that \eqref{SCC} is satisfied.
Note that in the degenerate case, $\theta=0$ the above estimates hold for $\log \rho.$ Indeed we have 
\begin{equation*}
\log \rho(x,t,\omega)\le \int_{\mathbb{T}} |\partial_x \log \rho| dx
\end{equation*}
and the claims follow by recalling Remark \ref{degenerate case}.
On the other hand we observe that 
\begin{equation}\label{rho L infty proof}
\rho^\alpha (x,t,\omega) \le 1 + \dfrac{\alpha}{| \alpha-\frac{1}{2}|}\int_{\mathbb{T}} \sqrt{\rho} |\partial_x \rho^{\alpha-\frac{1}{2}}| dx,
\end{equation}
hence, again by taking the sup on $\mathbb{T} \times (0,T \land \tau),$ the p-th power and the expectation in \eqref{rho L infty proof}, we use H\"{o}lder inequality to deduce \eqref{rho L infty}. To conclude we observe that by assuming that \eqref{NV} holds, then also negative powers of the density are controlled. Thus \eqref{1 div rho L infty} easily follows by \eqref{rho L infty beta} and \eqref{rho L infity alpha}.
\end{proof}
\noindent
For a better readability, we summarize the regularity for strong solutions obtained in the range given by  \eqref{SCC} and \eqref{NV}.
\begin{equation}
\begin{split}
& \rho \in L^p(\Omega; L^\infty (0,T\land \tau;H^1(\mathbb{T}))) \cap L^p(\Omega; L^2 (0,T\land \tau;H^2(\mathbb{T}))), \\ &  u \in L^p(\Omega; L^\infty (0,T\land \tau;L^2(\mathbb{T}))) \cap L^p(\Omega; L^2(0,T\land \tau;H^1 (\mathbb{T}))),  \\ & 
\rho \in L^p(\Omega; L^\infty (0,T\land \tau;L^\infty(\mathbb{T}))),  \; \,\frac{1}{\rho} \in L^p(\Omega; L^\infty (0,T\land \tau;L^\infty(\mathbb{T}))),
\end{split}
\end{equation}
for all $p \in [1,\infty).$
\subsection{High-order estimates}
The next part of this Section is devoted to the derivation of high-order estimates for Sobolev norms.  The key elements for the proof of these results are the positive lower bound for the density obtained in the \eqref{NV} regime and the time-integrability condition for the second derivative of $\rho$ given by Proposition \ref{BD prop}.
Also an appropriate change of variable will be used, together with a stochastic version of the Gronwall lemma. 
\\
\\
Because of the nonlinear structure of the Korteweg tensor, it is convenient to introduce the new variable 
\begin{equation*}
A(\rho):=\sqrt{\dfrac{k(\rho)}{\rho}} \partial_x \rho.
\end{equation*}
and the auxiliary quantity $\mu_k(\rho)$ satisfying $\mu'_k(\rho)= \sqrt{\rho k(\rho)}.$
In particular, system \eqref{main system} in the new variables reads 
\begin{equation}\label{eq A(rho)}
\partial_t A+ \partial_x ( A u)+ \partial_x (\mu'_k(\rho) \partial_x u)= 0
\end{equation}
and 
\begin{equation}\label{u systm}
du+ \big[u\partial_x u+ \dfrac{\partial_x p(\rho)}{\rho}\big]dt= \bigg[\dfrac{\partial_x (\mu(\rho)\partial_x u)}{\rho}+ \partial_x (\mu'_k(\rho) \partial_x A) + A \partial_x A \bigg]dt+ \mathbb{F}(\rho,u) dW.
\end{equation}
which is sometimes also referred as the augmented system.
Note that by virtue of the introduction of the new quantity $A(\rho)$ 
the high-order derivative terms appearing in \eqref{eq A(rho)}-\eqref{u systm} exhibit a skew-adjoint structure. Therefore a fundamental cancellation occurs in the $H^{s+1}(\mathbb{T}) \times H^{s}(\mathbb{T})$ estimates.
\\
\\
We start our analysis with the following proposition providing an $H^2 \times H^1$ bound. The application of the stochastic Gronwall Lemma, see Lemma 5.3 in \cite{Glatt-Holtz} yields the use of an auxiliary stopping time $\gamma^{(1)}_M$ appearing in the statement of the next proposition. The precise definition of $\gamma^{(1)}_M$ will be later given in \eqref{gamma1} and by virtue of further applications such as the extension argument given in Section \ref{Sec5}, we anticipate in the following proof that $\lim_{M \rightarrow \infty} \gamma^{(1)}_M= \infty \; a.s..$ 
\begin{proposition}\label{Prop h2 est}
Let $(\rho,u)$ be a solution of \eqref{main system}-\eqref{initial conditions}. Assume in addition that \eqref{SCC} and \eqref{NV} holds. Then, the following inequality is satisfied
\begin{equation}\label{H2 estimate}
\mathbb{E} \bigg[ \sup_{t \in [0,T \land \gamma^{(1)}_M \land \tau ]} \dfrac{1}{2}\bigg( \| \partial_{x} A(\rho) \|^2_{L^2} + \| \partial_x u \|^2_{L^2} \bigg) \bigg]+ \mathbb{E}\bigg[ \int_{0}^{T \land \gamma^{(1)}_M \land \tau} \int_{\mathbb{T}} \rho^{\alpha-1} |\partial_{xx} u|^2  dxdt \bigg] < +\infty.
\end{equation}
for a suitable stopping time $\gamma^{(1)}_M$ satisfying $\lim_{M \rightarrow \infty} \gamma^{(1)}_M= \infty \; a.s.$ 
\end{proposition}
\begin{proof}
We test \eqref{eq A(rho)} with $-\partial_{xx} A$ and after integrating by parts we deduce
\begin{equation}\label{dt A H1}
\dfrac{1}{2} \dfrac{d}{dt} \| \partial_x A\|^2_{L^2}= \int_{\mathbb{T}}\partial_x ( A u)\partial_{xx} Adx+ \int_{\mathbb{T}} \partial_x (\mu'_k(\rho) \partial_x u)\partial_{xx}Adx
\end{equation}
and similarly, we apply $\partial_x$ in the momentum equation \eqref{u systm} and then we use It\^{o} formula to the functional 
\begin{equation*}
F(u)= \dfrac{1}{2} \int_{\mathbb{T}} | \partial_x u |^2 dx
\end{equation*}
to deduce
\begin{equation}\label{dt u H1}
\begin{split}
& d\int_{\mathbb{T}} \dfrac{| \partial_x u|^2}{2} dx + \int_{\mathbb{T}}\partial_x \big[u\partial_x u+ \dfrac{\partial_x p(\rho)}{\rho}\big] \partial_x u \, dx dt= \int_{\mathbb{T}} \partial_x \bigg[\dfrac{\partial_x (\mu(\rho)\partial_x u)}{\rho} \bigg] \partial_x u dxdt \\ & + \int_{\mathbb{T}} \partial_x \bigg[ \partial_x (\mu'_k(\rho) \partial_x A) + A \partial_x A\bigg] \partial_x u \, dxdt + \int_{\mathbb{T}} \partial_x \mathbb{F}(\rho,u) \partial_x u \, dxdW \\ & + \dfrac{1}{2} \sum_{k=1}^{\infty} \int_{\mathbb{T}} | \partial_x F_k |^2 \, dxdt.
\end{split}
\end{equation}
We sum up \eqref{dt A H1} and \eqref{dt u H1} to deduce 
\begin{equation}\label{Hs equality intro}
\begin{split}
& \dfrac{1}{2} d \bigg( \| \partial_x A(\rho) \|^2_{L^2} + \| \partial_x u \|^2_{L^2}  \bigg) +  \int_{\mathbb{T}} \rho^{\alpha-1} | \partial_{xx} u |^2 dxdt = \int_{\mathbb{T}}  \partial_x ( A(\rho) u )\partial_{xx} A(\rho)dxdt \\ & -  \int_{\mathbb{T}} u \partial_x u \partial_{xx} u dxdt  -\int_{\mathbb{T}} \dfrac{\partial_x p(\rho) }{\rho} \partial_{xx} u dxdt - \int_{\mathbb{T}} \dfrac{\partial_x \mu(\rho) \partial_x u}{\rho} \partial_{xx} u dxdt \\ &  -\int_{\mathbb{T}} A \partial_x A \partial_{xx} u dxdt  + \int_{\mathbb{T}} \partial_x ( \mu'_k(\rho)\partial_x u) \partial_{xx} Adxdt - \int_{\mathbb{T}} \partial_x (\mu'_k(\rho) \partial_x A) \partial_{xx} u dxdt \\ & + \int_{\mathbb{T}} \partial_x \mathbb{F}(\rho,u) \partial_x u \, dxdW+ \dfrac{1}{2} \sum_{k=1}^{\infty} \int_{\mathbb{T}} | \partial_x F_k |^2 \, dxdt.
\end{split}
\end{equation}
Before starting the estimates of the right-hand-side of \eqref{Hs equality intro}, we note that thanks to the introduction of the new variable $A(\rho)$ an important cancellation between high order derivative terms occurs.  Specifically we have 
\begin{equation}\label{sum with cancellation}
\begin{split}
&  \int_{\mathbb{T}} \partial_x ( \mu'_k(\rho)\partial_x u) \partial_{xx} Adx - \int_{\mathbb{T}} \partial_x (\mu'_k(\rho) \partial_x A) \partial_{xx} u dx \\ & =  \int_{\mathbb{T}} \partial_x \mu'_k(\rho) \partial_x u \partial_{xx} Adx- \int_{\mathbb{T}} \partial_x \mu'_k(\rho) \partial_x A \partial_{xx} udx.
\end{split}
\end{equation}
Thus we end up with the following equality 
\begin{equation}\label{Hs equality final}
\begin{split}
& \dfrac{1}{2} d \bigg( \| \partial_x A(\rho) \|^2_{L^2} + \| \partial_x u \|^2_{L^2}  \bigg) +  \int_{\mathbb{T}} \rho^{\alpha-1} | \partial_{xx} u |^2 dxdt =  \int_{\mathbb{T}}  \partial_x ( A(\rho) u )\partial_{xx} A(\rho)dxdt \\ & -  \int_{\mathbb{T}} u \partial_x u \partial_{xx} u dxdt  -\int_{\mathbb{T}} \dfrac{\partial_x p(\rho) }{\rho} \partial_{xx} u dxdt - \int_{\mathbb{T}} \dfrac{\partial_x \mu(\rho) \partial_x u}{\rho} \partial_{xx} u dxdt \\ & -\int_{\mathbb{T}} A \partial_x A \partial_{xx} u dxdt + \int_{\mathbb{T}} \partial_x \mu'_k(\rho) \partial_x u \partial_{xx} A dxdt  - \int_{\mathbb{T}} \partial_x \mu'_k(\rho) \partial_x A \partial_{xx} udxdt  \\ & + \int_{\mathbb{T}} \partial_x \mathbb{F}(\rho,u) \partial_x u \, dxdW+ \dfrac{1}{2} \sum_{k=1}^{\infty} \int_{\mathbb{T}} | \partial_x F_k |^2 \, dxdt = \sum_{i=1}^{9} I_i.
\end{split}
\end{equation}
Now we start estimating the integrals in the right hand side of \eqref{Hs equality final}. In particular, we will make an extensive use of H\"{o}lder, Sobolev, Poincarè and Young inequalities. The explicit dependence on positive constants depending on $\alpha$ and $\beta$ is omitted. 
\begin{equation}\label{I1 s=1}
\begin{split}
|I_1|& \le \int_{\mathbb{T}} | \partial_x A|^2 |\partial_x u| dx +\int_{\mathbb{T}} |A| | \partial_x A| | \partial_{xx}u| dx \\ & \le \| \partial_x A \|^2_{L^2} \| \partial_x u \|_{L^\infty} + \| A \|_{L^\infty} \| \frac{1}{\rho^{\frac{\alpha-1}{2}}}\|_{L^\infty} \| \partial_x A \|_{L^2} \| \rho^{\frac{\alpha-1}{2}} \partial_{xx} u \|_{L^2} \\ & \le \| \partial_x A(\rho) \|_{L^2} \| \partial_{xx} \rho^{\frac{\alpha+\beta+1}{2}} \|_{L^2} \| \partial_{xx} u \|_{L^2} \| \rho^{-\frac{\alpha}{2}} \|_{L^\infty}  \\ & + \| \partial_x A(\rho) \|^2_{L^2} \| \frac{1}{\rho^{\frac{\alpha-1}{2}}}\|_{L^\infty} \| \rho^{\frac{\alpha-1}{2}} \partial_{xx} u \|_{L^2} \\ & \le \delta \| \rho^{\frac{\alpha-1}{2}} \partial_{xx} u \|^2_{L^2}+ C(\delta) \| \partial_x A(\rho) \|^2_{L^2} \| \partial_{xx} \rho^{\frac{\alpha+\beta+1}{2}} \|^2_{L^2} \| \rho^{\frac{-2 \alpha+1}{2}} \|^2_{L^\infty},
\end{split}
\end{equation}
where we have also used that up to a positive multiplicative constant 
\begin{equation*}
|\partial_x A| \le \rho^{-\frac{\alpha}{2}} |\partial_{xx} \rho^{\frac{\alpha+\beta+1}{2}} |+\rho^{-\frac{\alpha}{2}}| \partial_x \rho^{\frac{\alpha+\beta+1}{4}}|^2
\end{equation*}
and therefore by using also Lemma \ref{Functional ineq } with $f=\rho^{\frac{\alpha+\beta+1}{2}},$ we have 
\begin{equation*}
\| \partial_x A \|_{L^2} \le \| \rho^{-\frac{\alpha}{2}} \|_{L^\infty} \| \rho^{\frac{\alpha}{2}} \partial_x A \|_{L^2}\le \| \rho^{-\frac{\alpha}{2}} \|_{L^\infty} \| \partial_{xx} \rho^{\frac{\alpha+\beta+1}{2}} \|_{L^2}.
\end{equation*}
With similar lines of argument we deduce 
\begin{equation}
\begin{split}
|I_2| & \le \| u \|_{L^\infty} \| \partial_x u \|_{L^2} \| \rho^{\frac{\alpha-1}{2}}\partial_{xx} u \|_{L^2} \| \frac{1}{\rho^{\frac{\alpha-1}{2}}}\|_{L^\infty}  \\ & \le C(\delta) \| u \|^2_{L^\infty} \| \partial_x u \|^2_{L^2}\| \frac{1}{\rho^{\frac{\alpha-1}{2}}}\|^2_{L^\infty}+ \delta \| \rho^{\frac{\alpha-1}{2}}\partial_{xx} u \|^2_{L^2} \\ & \le C(\delta) \| \partial_x u \|^2_{L^2} \| \frac{1}{\rho^{\frac{\alpha-1}{2}}} \|^2_{L^\infty}( \| u \|^2_{L^2}+ \| \partial_x u \|^2_{L^2})+ \delta \| \rho^{\frac{\alpha-1}{2}}\partial_{xx} u \|^2_{L^2},
\end{split}
\end{equation}
\\
\begin{equation}
\begin{split}
|I_3| &  \le c \| \rho^{\frac{2\gamma-\alpha-\beta-2}{2}} \|_{L^\infty} \| \partial_x A \|_{L^2} \| \rho^{\frac{\alpha-1}{2}}\partial_{xx} u \|_{L^2} \\ & \le C(\delta) \| \rho^{\frac{2\gamma-\alpha-\beta-2}{2}} \|^2_{L^\infty}  \| \partial_x A \|^2_{L^2}+\delta \| \rho^{\frac{\alpha-1}{2}}\partial_{xx} u \|^2_{L^2},
\end{split}
\end{equation}
while the viscous term is handled as follows
\begin{equation}
\begin{split}
| I_4| & \le C \| \rho^{\frac{-\beta-2}{2}} \|_{L^\infty} \| \partial_x \rho^{\frac{\alpha+\beta+1}{2}} \|_{L^\infty} \| \partial_x u \|_{L^2} \| \rho^{\frac{\alpha-1}{2}}\partial_{xx} u \|_{L^2} \\ & \le C(\delta) \| \rho^{\frac{-\beta-2}{2}} \|^2_{L^\infty} \| \partial_{xx} \rho^{\frac{\alpha+\beta+1}{2}} \|^2_{L^2} \| \partial_x u \|^2_{L^2} + \delta \| \rho^{\frac{\alpha-1}{2}}\partial_{xx} u \|^2_{L^2}.
\end{split}
\end{equation}
Finally, again with an extensive use of H\"{o}lder, Young and Sobolev inequalities, the capillarity terms give the following contributions
\begin{equation}
\begin{split}
| I_5 | & \le \delta \| \rho^{\frac{\alpha-1}{2}} \partial_{xx} u \|^2_{L^2}+ C(\delta) \| \partial_x A(\rho) \|^2_{L^2} \| \partial_{xx} \rho^{\frac{\alpha+\beta+1}{2}} \|^2_{L^2} \| \rho^{\frac{-2 \alpha+1}{2}} \|^2_{L^\infty},
\end{split}
\end{equation}
while similarly to the estimate of $I_1,$ since $\partial_x \mu'_k(\rho)= \big( \frac{\beta+1}{2} \big) A$ we have  
\begin{equation}
\begin{split}
| I_6 | & = \bigg| \bigg(\frac{\beta+1}{2}\bigg) \int_{\mathbb{T}} \partial_x A \partial_x u \partial_{xx} A dx \bigg|  \\ & \le  \delta \| \rho^{\frac{\alpha-1}{2}} \partial_{xx} u \|^2_{L^2}+C(\delta)\| \partial_x A(\rho) \|^2_{L^2} \| \partial_{xx} \rho^{\frac{\alpha+\beta+1}{2}} \|^2_{L^2} \| \rho^{\frac{-2\alpha+1}{2}} \|^2_{L^\infty}
\end{split}
\end{equation}
and 
\begin{equation}
\begin{split}
| I_7 | & = \bigg| \bigg(\frac{\beta+1}{2}\bigg) \int_{\mathbb{T}} \partial_x A \partial_x A \partial_{xx} u dx \bigg|  \\ & \le  \delta \| \rho^{\frac{\alpha-1}{2}} \partial_{xx} u \|^2_{L^2}+C(\delta)\| \partial_x A(\rho) \|^2_{L^2} \| \partial_{xx} \rho^{\frac{\alpha+\beta+1}{2}} \|^2_{L^2} \| \rho^{\frac{-2\alpha+1}{2}} \|^2_{L^\infty}
\end{split}
\end{equation}
Furthermore, we use \eqref{F1}-\eqref{F2} to handle the It\^{o} correction terms as follows
\begin{equation}\label{I9 s=1}
|I_9| = \dfrac{1}{2} \sum_{k=1}^{\infty} \int_{\mathbb{T}} | \partial_x F_k |^2 dx \lesssim \sum_{k=1}^{\infty} f^2_k \le C.
\end{equation}
Collecting \eqref{I1 s=1}-\eqref{I9 s=1} together and choosing properly the constant $\delta$ we have 
\begin{equation}
\begin{split}
& \dfrac{1}{2} d \bigg( \| \partial_x A(\rho) \|^2_{L^2} + \| \partial_x u \|^2_{L^2}  \bigg) +  c\int_{\mathbb{T}} \rho^{\alpha-1} | \partial_{xx} u |^2 dx \\ & \le a(t) \bigg( \| \partial_x A(\rho) \|^2_{L^2} + \| \partial_x u \|^2_{L^2}  \bigg)+b(t)+ |I_8|,
\end{split}
\end{equation}
for a constant $c>0$ and $a(t), \; b(t) \in L^p(\Omega;L^1(0,T)).$  for all $p \in [1, \infty).$
Specifically we have 
\begin{equation*}
\begin{split}
& a(t)  = \big( \| \rho^{\frac{-\beta-2}{2}} \|^2_{L^2} + \| \rho^{\frac{-2\alpha+1}{2}} \|^2_{L^\infty} \big)\| \partial_{xx} \rho^{\frac{\alpha+\beta+1}{2}} \|^2_{L^2} + \| \frac{1}{\rho^{\frac{\alpha-1}{2}}} \|^2_{L^\infty} \big( \| \partial_x u \|^2_{L^2} + \| u \|^2_{L^2} \big) + \| \rho^{\frac{2\gamma-\alpha-\beta-2}{2}} \|^2_{L^\infty}\\ &
b(t)=C.
\end{split}
\end{equation*}
Note that the estimate of the term $|I_8|$ involves a stochastic integral which we handle in expectation. To this purpose we first perform a localization argument by introducing the following stopping time
\begin{equation}\label{gamma1}
\gamma^{(1)}_M= \inf_{ t \ge 0} \bigg\{ \int_{0}^{t} a(s)ds > M \bigg\}, \quad \inf \emptyset= +\infty,
\end{equation}
 so that we can apply a stochastic version of the Gronwall lemma. In particular,  we have that 
\begin{equation}
\int_{0}^{\gamma^{(1)}_M} a(t) dt \le M, \quad a.s.
\end{equation}
and given a pair of stopping times $\tau_a, \; \tau_b$ such that $0 \le \tau_a < \tau_b \le \gamma^{(1)}_M \land T$ we integrate in $[\tau_a, t]$ and taking the sup in $[\tau_a, \tau_b]$ and the expectation we deduce
\begin{equation}
\begin{split}
& \mathbb{E} \bigg[ \sup_{\tau_a \le t \le \tau_b} \dfrac{1}{2}  \bigg( \| \partial_x A(\rho) \|^2_{L^2} + \| \partial_x u \|^2_{L^2} \bigg) \bigg] +  c \mathbb{E} \bigg[ \int_{\tau_a}^{\tau_b}\int_{\mathbb{T}} \rho^{\alpha-1} | \partial_{xx} u |^2 dxdt \bigg] \\ & \le \mathbb{E} \bigg[ \| \partial_x A(\rho) (\tau_a) \|^2_{L^2} + \| \partial_x u (\tau_a) \|^2_{L^2} \bigg] +c \mathbb{E} \bigg[ \int_{\tau_a}^{\tau_b} a(t) \bigg( \| \partial_x A(\rho) \|^2_{L^2} + \| \partial_x u \|^2_{L^2}  \bigg) dt \bigg] \\ &  + \mathbb{E} \int_{\tau_a}^{\tau_b} b(t) dt +  \mathbb{E} \sup_{\tau_a \le t \le \tau_b} \big| \int_{\tau_a}^{s} I_8 \big|.
\end{split}
\end{equation}
Now we observe that again by virtue of \eqref{F1}-\eqref{F2} and the Burkholder-Davis-Gundy inequality we have
\begin{equation}\label{est stoch int H2}
\begin{split}
& \mathbb{E} \bigg[ \sup_{\tau_a \le t \le \tau_b} \bigg| \int_{\tau_a}^{s} \int_{\mathbb{T}} \partial_x \mathbb{F}(\rho,u) \partial_x u dxdW \bigg| \bigg] \lesssim \mathbb{E} \bigg[ \int_{\tau_a}^{\tau_b} \sum_{k=1}^{\infty} \bigg| \int_{\mathbb{T}} \partial_x F_k \partial_x u dx\bigg|^2 dt \bigg]^{\frac{1}{2}} \\ & \lesssim  \mathbb{E} \bigg[ \int_{\tau_a}^{\tau_b} \sum_{k=1}^{\infty} f^2_k \| \partial_x u \|^2_{L^2} dt \bigg]^{\frac{1}{2}} \lesssim \mathbb{E} \bigg[ \int_{\tau_a}^{\tau_b} \| \partial_x u \|^2_{L^2} dt \bigg]^{\frac{1}{2}} \le \mathbb{E} \bigg[ 1+ \int_{\tau_a}^{\tau_b} \| \partial_x u \|^2_{L^2} dt \bigg].
\end{split}
\end{equation}
Thus, by virtue of \eqref{est stoch int H2} we deduce that 
\begin{equation*}
\begin{split}
& \mathbb{E} \bigg[ \sup_{\tau_a \le t \le \tau_b} \dfrac{1}{2}  \bigg( \| \partial_x A(\rho) \|^2_{L^2} + \| \partial_x \tilde{u} \|^2_{L^2} \bigg) \bigg] +  c \mathbb{E} \bigg[ \int_{\tau_a}^{\tau_b}\int_{\mathbb{T}} \rho^{\alpha-1} | \partial_{xx} \tilde{u} |^2 dxdt \bigg] \\ & \le \mathbb{E} \bigg[ \| \partial_x A(\rho) (\tau_a) \|^2_{L^2} + \| \partial_x u (\tau_a) \|^2_{L^2} \bigg] +c \mathbb{E} \bigg[ \int_{\tau_a}^{\tau_b} a(t) \bigg( \| \partial_x A(\rho) \|^2_{L^2} + \| \partial_x u \|^2_{L^2}  \bigg) dt \bigg] \\ &  + \mathbb{E} \int_{\tau_a}^{\tau_b} b(t) dt,
\end{split}
\end{equation*}
from which we apply the stochastic Gronwall lemma,  Lemma 5.3 in \cite{Glatt-Holtz} to infer 
\begin{equation*}
\begin{split}
& \mathbb{E} \bigg[ \sup_{0 \le t \le T \land \gamma^{(1)}_M \land \tau } \dfrac{1}{2}  \bigg( \| \partial_x A(\rho) \|^2_{L^2} + \| \partial_x u \|^2_{L^2} \bigg) \bigg] +  c \mathbb{E} \bigg[ \int_{0}^{T \land \gamma^{(1)}_M \land \tau}\int_{\mathbb{T}} \rho^{\alpha-1} | \partial_{xx} u |^2 dxdt \bigg] < + \infty.
\end{split}
\end{equation*}
The next step is to prove that $\lim_{M \rightarrow \infty} \gamma^{(1)}_M= \infty \; a.s.$ We first observe that by Markov inequality 
\begin{equation*}
\mathbb{P} \bigg( \gamma^{(1)}_M \le T \bigg) \le \mathbb{P} \bigg( \int_{0}^{T} a(t)dt \ge M \bigg) \le \dfrac{1}{M} \mathbb{E} \bigg( \int_{0}^{T} a(t) dt \bigg),
\end{equation*}
and by virtue of Proposition \ref{energy prop}, Proposition \ref{BD prop} and Proposition \ref{prop:vacuum} we have 
\begin{equation}\label{markov gamma1}
\mathbb{E} \bigg( \int_{0}^{T} a(t) dt \bigg) \le C. 
\end{equation}
Therefore our claim follows by sending $ M \rightarrow \infty$ in \eqref{markov gamma1}. This concludes the proof of Proposition \ref{Prop h2 est}.
\end{proof}
\noindent
We conclude this Section with the following $H^{s+1}(\mathbb{T}) \times H^{s}(\mathbb{T})$ estimate.  Note that the exponent $s$ is chosen in such a way that $s > \frac{n}{2}+2= \frac{5}{2},$ which is enough to control the norms of the blow-up condition \eqref{limit cond}.
\begin{proposition}\label{Prop global s+1 s}
Under the same assumptions of Proposition \ref{Prop h2 est}, the following inequality holds
\begin{equation}\label{p=1 global hs}
\begin{split}
& \mathbb{E} \bigg[ \sup_{t \in [0, T \land \gamma^{(3)}_M \land \tau] } \dfrac{1}{2}\bigg( \| A(\rho) \|^2_{H^3}+ \| u \|^2_{H^3} \bigg)  \bigg] \\ & + \mathbb{E} \bigg[ \int_{0}^{T \land \gamma^{(3)}_M \land \tau} \int_{\mathbb{T}} \rho^{\alpha-1} (|\partial_x u |^2+|\partial_{xx} u |^2+|\partial^3_x u |^2+|\partial^4_x u |^2) dt  \bigg] < \infty.
\end{split}
\end{equation}
for a suitable stopping time $\gamma^{(3)}_M$ satisfying $\lim_{M \to \infty} \gamma^{(3)}_M= \infty$ a.s..
\end{proposition}
\begin{proof}
As a first step of the proof, we deduce an $H^3(\mathbb{T})\times H^2(\mathbb{T})$ estimates. In particular, our goal is to prove that 
\begin{equation}\label{eq:s2}
\begin{split}
& \partial_{xx} A(\rho) \in L^1(\Omega;L^\infty(0,T  \land \tau;L^2(\mathbb{T}))),\quad \partial_{xx} u \in L^1(\Omega;L^\infty(0,T \land \tau;L^2(\mathbb{T}))), \\ & 
\rho^{\frac{\alpha-1}{2}}\partial^3_{x} u \in L^1(\Omega;L^2(0,T \land \tau;L^2(\mathbb{T}))).
\end{split}
\end{equation}
The proof mimic the same lines of arguments with respect to the proof of Proposition \ref{Prop h2 est}. In particular, we apply $\partial_{xx}$ to \eqref{eq A(rho)} and we multiply it by $\partial_{xx} A(\rho).$ Then, after integrating by parts we have 
\begin{equation}\label{ito rho s=2}
\dfrac{d}{dt} \int_{\mathbb{T}} \frac{ | \partial_{xx} A|^2}{2} dx+ \int_{\mathbb{T}} \partial^3_{x}( A u)\partial_{xx} A dx+ \int_{\mathbb{T}}\partial^3_x( \mu'_k(\rho) \partial_x u) \partial_{xx} A dx =0.
\end{equation}
Then we apply $\partial_{xx}$ to  \eqref{u systm}  and by using It\^{o} formula to the functional $F(u)=\frac{1}{2}\int_{\mathbb{T}} |\partial_{xx} u|^2$
we sum up with \eqref{ito rho s=2} and integrate by parts to deduce
\begin{equation}\label{sum with cancellation s=2}
\begin{split}
& \dfrac{1}{2} d \bigg( \| \partial_{xx} A \|^2_{L^2} + \| \partial_{xx} u \|^2_{L^2}  \bigg) +  \int_{\mathbb{T}} \rho^{\alpha-1} | \partial^3_{x} u |^2 dxdt = -  \int_{\mathbb{T}}  \partial^3_x ( A u )\partial_{xx} Adxdt \\ & - \int_{\mathbb{T}} \partial^3_x( \mu'_k(\rho) \partial_x u ) \partial_{xx} A(\rho) dxdt -\int_{\mathbb{T}} \partial_{xx}(u \partial_x u) \partial_{xx} u dxdt  -\int_{\mathbb{T}} \partial_{xx} \bigg( \dfrac{\partial_x p(\rho) }{\rho} \bigg) \partial_{xx} u dxdt \\ & + \int_{\mathbb{T}} \partial_{xx} \bigg( \dfrac{\partial_x \mu(\rho) \partial_x u}{\rho} \bigg) \partial_{xx} u dxdt - \int_{\mathbb{T}} | \partial_{xx} u |^2 \partial_x \bigg( \frac{\mu(\rho)}{\rho} \bigg)dxdt  + \int_{\mathbb{T}} \partial^3_x ( \mu'_k(\rho) \partial_x A) \partial_{xx} u dxdt  \\ & + \int_{\mathbb{T}} \partial_{xx} (A \partial_x A ) \partial_{xx} u dxdt + \int_{\mathbb{T}} \partial_{xx} \mathbb{F}(\rho,u) \partial_{xx} u \, dxdW + \dfrac{1}{2} \sum_{k=1}^{\infty} \int_{\mathbb{T}} | \partial_{xx}F_k |^2 \, dxdt.
\end{split}
\end{equation}
Now, similarly to the $H^2(\mathbb{T}) \times H^1(\mathbb{T})$ estimate, we highlight the following cancellation
\begin{equation}
\begin{split}
& -\int_{\mathbb{T}} \partial^3_x( \mu'_k(\rho) \partial_x u ) \partial_{xx} A(\rho) dx+ \int_{\mathbb{T}} \partial^3_x ( \mu'_k(\rho) \partial_x A) \partial_{xx} u dx \\ & = \int_{\mathbb{T}} \bigg( \frac{\beta+1}{2} \bigg) \partial_x A \partial_x u \partial^3_x A dx+ \int_{\mathbb{T}} (\beta+1) A \partial_{xx} u \partial^3_x A dx \\ & -\int_{\mathbb{T}} \bigg( \frac{\beta+1}{2} \bigg) |\partial_x A|^2  \partial^3_x u dx- \int_{\mathbb{T}} (\beta+1) A \partial_{xx} A \partial^3_x u dx \\ & = -\int_{\mathbb{T}} \bigg( \frac{\beta+1}{2} \bigg) |\partial_{xx} A|^2 \partial_x u dx- \int_{\mathbb{T}} \bigg(\frac{\beta+1}{2}\bigg) \partial_x A \partial_{xx} A \partial_{xx} u dx -\int_{\mathbb{T}} (2\beta+2) A \partial_{xx} A \partial^3_x u dx,
\end{split}
\end{equation}
thus \eqref{sum with cancellation s=2} reduces to 
\begin{equation}\label{s=2}
\begin{split}
& \dfrac{1}{2} d \bigg( \| \partial_{xx} A(\rho) \|^2_{L^2} + \| \partial_{xx} u \|^2_{L^2}  \bigg) +  \int_{\mathbb{T}} \rho^{\alpha-1} | \partial^3_{x} u |^2 dxdt = -  \int_{\mathbb{T}}  \partial^3_x ( A(\rho) u )\partial_{xx} A(\rho)dxdt \\ &  -\int_{\mathbb{T}} \partial_{xx}(u \partial_x u) \partial_{xx} u dxdt  -\int_{\mathbb{T}} \partial_{xx} \bigg( \dfrac{\partial_x p(\rho) }{\rho} \bigg) \partial_{xx} u dxdt + \int_{\mathbb{T}} \partial_{xx} \bigg( \dfrac{\partial_x \mu(\rho) \partial_x u}{\rho} \bigg) \partial_{xx} u dxdt \\ & - \int_{\mathbb{T}} \partial_x \bigg( \frac{\mu(\rho)}{\rho} \bigg) \partial_{xx} u \partial^3_x u dxdt -\int_{\mathbb{T}} \bigg( \frac{\beta+1}{2} \bigg) |\partial_{xx} A|^2 \partial_x u dxdt- \int_{\mathbb{T}} \bigg(\frac{\beta+1}{2}\bigg) \partial_x A \partial_{xx} A \partial_{xx} u dxdt \\ & -\int_{\mathbb{T}} (2\beta+2) A \partial_{xx} A \partial^3_x u dxdt + \int_{\mathbb{T}} \partial_{xx} (A \partial_x A ) \partial_{xx} u dxdt+ \int_{\mathbb{T}} \partial_{xx} \mathbb{F}(\rho,u) \partial_{xx} u \, dxdW \\ & + \dfrac{1}{2} \sum_{k=1}^{\infty} \int_{\mathbb{T}} | \partial_{xx}F_k |^2 \, dxdt = \sum_{i=1}^{11} I_i.
\end{split}
\end{equation}
We start estimating the right hand side of \eqref{s=2}.  In order to avoid heavy notations,  we make use of a general density exponent $m_1=m_1(\alpha,\beta,\gamma),$ which can vary from lines to lines and can be computed explicitly. This is not restrictive since both $\rho, \frac{1}{\rho}$ are controlled in $L^\infty_{t,x}.$ 
\begin{equation}\label{I_1 s=2}
\begin{split}
|I_1| & \le \bigg| \int_{\mathbb{T}} | \partial_{xx} A|^2 \partial_x u + \partial_x A \partial_{xx} u \partial_{xx} A + A \partial^3_x u \partial_{xx} A dx \bigg| \\ & \le \| \partial_{xx} A \|^2_{L^2} ( \| \rho^{\frac{\alpha-1}{2}} \partial_{xx} u \|^2_{L^2} + \| \rho^{m_1}\|^2_{L^\infty} +C(\delta) \| A \|^2_{L^\infty} \| \rho^{m_1} \|^2_{L^\infty} ) + \delta \| \rho^{\frac{\alpha-1}{2}} \partial^3_x u \|^2_{L^2}
\end{split}
\end{equation}
\begin{equation}
\begin{split}
|I_2| &  \le \| \partial_{xx} u \|^2_{L^2} \big( \| \rho^{\frac{\alpha-1}{2}} \partial_{xx} u \|^2_{L^2} + \| \rho^{m_1}\|^2_{L^\infty} + C(\delta) \|u\|^2_{L^\infty} \| \rho^{m_1} \|^2_{L^\infty} \big) + \delta \| \rho^{\frac{\alpha-1}{2}} \partial^3_x u \|^2_{L^2}
\end{split}
\end{equation}
and again by using H\"{o}lder, Sobolev and Young inequalities 
\begin{equation}
| I_3| \le \| \partial_{xx} u \|^2_{L^2}+ \| \partial_x A \|^2_{L^2} \| \rho^{m_1} \|^2_{L^\infty}+ \| \partial_{xx} u \|^2_{L^2} + \| \partial_{xx} A \|^2_{L^2} \| \rho^{m_1} \|^2_{L^\infty},
\end{equation}
\begin{equation}
\begin{split}
|I_4| & \le  \| \partial_{xx} A \|^2_{L^2} + \| \partial_{xx} u \|^2_{L^2} (\| \rho^{\frac{\alpha-1}{2}} \partial_{xx} u \|^2_{L^2} + \| \rho^{m_1}\|^2_{L^\infty} ) \\ & + C(\delta) \| \rho^{m_1}\|^2_{L^\infty} \| \partial_x A \|^2_{L^2} + \delta \| \rho^{\frac{\alpha-1}{2}} \partial^3_x u \|^2_{L^2},
\end{split}
\end{equation}
\begin{equation}
|I_5| \le C(\delta) \| \partial_{xx} u \|^2_{L^2} \| A \|^2_{L^\infty} \| \rho^{m_1}\|^2_{L^\infty} + \delta \| \rho^{\frac{\alpha-1}{2}} \partial^3_x u \|^2_{L^2} ,
\end{equation}
while the capillarity terms can be estimated as follows 
\begin{equation}
|I_6|+ |I_7| \le \| \partial_{xx} A \|^2_{L^2} (\| \rho^{\frac{\alpha-1}{2}} \partial_{xx} u \|^2_{L^2} + \| \rho^{m_1}\|^2_{L^\infty} )
\end{equation}
and 
\begin{equation}
|I_8| \le C(\delta) \| A\|^2_{L^\infty} \| \partial_{xx} A \|^2_{L^2} \| \rho^{m_1}\|^2_{L^\infty}+ \delta \| \rho^{\frac{\alpha-1}{2}} \partial^3_x u \|^2_{L^\infty},
\end{equation}
\begin{equation}
|I_9| \le \| \partial_{xx} A \|^2_{L^2} \| \partial_x A \|_{L^\infty} + \| \partial_{xx} u \|^2_{L^2}+ \| A \|^2_{L^\infty} \| \rho^{m_1}\|^2_{L^\infty} \| \partial_{xx} A \|^2_{L^2}+ \delta \| \rho^{\frac{\alpha-1}{2}} \partial_{xx} u \|^2_{L^2}
\end{equation}
Similarly to the proof of Proposition \ref{Prop h2 est}, the contributions of the stochastic force is given by 
\begin{equation}\label{I_8 s=2}
\begin{split}
|I_{11}| &=\frac{1}{2} \sum_{k=1}^{\infty} \int_{\mathbb{T}} |\partial_{xx} F_k|^2 dx \le c \sum_{k=1}^{\infty} \| \partial_{xx} F_k \|^2_{L^\infty} \le c \sum_{k=1}^{\infty} {f_k}^2 \le C.
\end{split}
\end{equation}
Hence, collecting the estimates \eqref{I_1 s=2}-\eqref{I_8 s=2}, we choose properly the Young constant $\delta$ to get 
\begin{equation}
\begin{split}
& \dfrac{1}{2}\text{d} ( \| \partial_{xx} A \|^2_{L^2} + \| \partial_{xx} u \|^2_{L^2} ) + c\int_{\mathbb{T}} \rho^{\alpha-1} | \partial^3_x u |^2 dx \\ &  \le \bigg(\| \partial_{xx} A \|^2_{L^2} + \| \partial_{xx} u \|^2_{L^2}\bigg)a^{(2)}(t)+b^{(2)}(t)+I_{10},
\end{split}
\end{equation}
with again $c>0$ and $a^{(2)}(t) \in L^1(\Omega; L^1(0,T)), \, b^{(2)}(t) \in L^p(\Omega;L^1(0,T))$ defined by 
\begin{equation*}
\begin{split}
& a^{(2)}(t)= \| \rho^{\frac{\alpha-1}{2}} \partial_{xx} u \|^2_{L^2} + \| \rho^{m_1}\|^2_{L^\infty} +  \| A \|^2_{L^\infty} \| \rho^{m_1} \|^2_{L^\infty} +  \|u\|^2_{L^\infty} \| \rho^{m_1} \|^2_{L^\infty} + \| A \|_{L^\infty}+ C, \\ & 
b^{(2)}(t)=\| \partial_x A \|^2_{L^2} \| \rho^{m_1} \|^2_{L^\infty}+ C
\end{split}
\end{equation*}
Also in this case we first perform a localization argument. We define the stopping time
\begin{equation}\label{gamma2}
\gamma^{(2)}_M= \inf \big\{ t\ge 0 \; :  \int_{0}^{t \land \tau} a^{(2)}(s)ds > M \big\}
\end{equation}
and integrating in time and applying the expectation we deduce 
\begin{equation}
\begin{split}
& \mathbb{E} \bigg[ \sup_{\tau_a \le t \le \tau_b} \frac{1}{2} \bigg( \| \partial_{xx} A(\rho)\|^2_{L^2} + \| \partial_{xx} u \|^2_{L^2} \bigg) \bigg]+ c\mathbb{E} \bigg[ \int_{\tau_a}^{\tau_b}\int_{\mathbb{T}} \rho^{\alpha-1} | \partial^3_x u |^2 dxdt \bigg] \\ &  \le \mathbb{E} ( \| \partial_{xx} A(\rho(\tau_a)) \|^2_{L^2} + \| \partial_{xx} u (\tau_a) \|^2_{L^2} )+ \mathbb{E} \int_{\tau_a}^{\tau_b}(\| \partial_{xx} A(\rho) \|^2_{L^2} + \| \partial_{xx} u \|^2_{L^2})a^{(2)}(t) \\ & +\mathbb{E} \int_{\tau_a}^{\tau_b}b^{(2)}(t) + \mathbb{E} \sup_{ \tau_a \le s \le \tau_b} | \int_{\tau_a}^{t}I_{10}|.
\end{split}
\end{equation}
for $\tau_a$ and $\tau_b$ being two stopping times such that $\tau_a < t < \tau_b < T \land \gamma^{(2)}_M.$
To conclude the estimate we observe that 
\begin{equation*}
\begin{split}
\mathbb{E} \bigg[\sup_{ \tau_a \le s \le \tau_b}  \bigg| {\int_{0}^{s} \int_{\mathbb{T}} \partial_{xx} \mathbb{F}(\rho,u)\partial_{xx} udxdW \bigg| } \bigg] & \lesssim \mathbb{E} \bigg[ \int_{\tau_a}^{\tau_b} C_1 \| \partial_{xx} u \|^2_{L^2} \bigg]^{\frac{1}{2}} \le \mathbb{E} \bigg[ 1+ \int_{\tau_a}^{\tau_b} C_2 \| \partial_{xx} u \|^2_{L^2} \bigg],
\end{split}
\end{equation*}
for two positive constants $C1$ and $C_2.$
Therefore we end up with the following inequality
\begin{equation}\label{befor gronw gamma2}
\begin{split}
& \mathbb{E} \bigg[ \sup_{ \tau_a \le s \le \tau_b} \frac{1}{2}( \| \partial_{xx} A(\rho) \|^2_{L^2} + \| \partial_{xx} u \|^2_{L^2} )\bigg] + c\mathbb{E} \bigg[ \int_{\tau_a}^{\tau_b}\int_{\mathbb{T}} \rho^{\alpha-1} | \partial^3_x u |^2 dxdt \bigg] \\ & \le \mathbb{E} ( \| \partial_{xx} A(\rho(\tau_a)) \|^2_{L^2} + \| \partial_{xx} u (\tau_a) \|^2_{L^2} )+ \mathbb{E} \int_{\tau_a}^{\tau_b}(\| \partial_{xx} A(\rho) \|^2_{L^2} + \| \partial_{xx} u \|^2_{L^2})a^{(2)}(t)dt \\ & +\mathbb{E} \int_{\tau_a}^{\tau_b}b^{(2)}(t)dt.
\end{split}
\end{equation}
Finally, by means of the stochastic Gronwall Lemma we deduce
\begin{equation}
\mathbb{E} \bigg(\sup_{  s \in [0, T \land \gamma^{(2)}_M \land \tau] }\dfrac{1}{2}( \| \partial_{xx} A(\rho) \|^2_{L^2}+ \| \partial_{xx} u \|^2_{L^2}) + \int_{0}^{T \land \gamma^{(2)}_M \land \tau}\int_{\mathbb{T}} \rho^{\alpha-1} | \partial^3_x u|^2 dxds \bigg) < +\infty
\end{equation}
\\
and since by virtue of Proposition \ref{energy prop}, Proposition \ref{BD prop}, Proposition \ref{prop:vacuum} and Proposition \ref{H2 estimate} 
\begin{equation}\label{expectation a2}
\mathbb{E} \int_{0}^{T \land \tau} a^{(2)}(t) < \infty ,
\end{equation}
then by using Markov inequality 
\begin{equation}\label{markov M s=2}
\mathbb{P} \big( \gamma^{(2)}_M \le T \big) \le\mathbb{P} \bigg( \int_{0}^{T \land \tau} a^{(2)}(t)dt \ge M \bigg) \le \frac{1}{M} \mathbb{E} \bigg( \int_{0}^{T \land \tau} a^{(2)}(t)dt \bigg) 
\end{equation}
which implies $\lim_{M \to \infty} \gamma^{(2)}_M= \infty$ a.s..
Note that the key element in \eqref{expectation a2} is the control of $\rho^{\frac{\alpha-1}{2}} \partial_{xx} u  \in L^1_wL^2_tL^2_x$ which is guaranteed by virtue of Proposition \ref{H2 estimate}. 
This concludes the proof of \eqref{eq:s2}.  
\\
\\
The final step is to deduce the following regularity
\begin{equation}\label{eq:s3}
\begin{split}
& \partial^3_{x} A(\rho) \in L^1(\Omega;L^\infty(0,T  \land \tau;L^2(\mathbb{T}))),\quad \partial^3_{x} u \in L^1(\Omega;L^\infty(0,T \land \tau;L^2(\mathbb{T}))), \\ & 
\rho^{\frac{\alpha-1}{2}}\partial^4_{x} u \in L^1(\Omega;L^2(0,T \land \tau;L^2(\mathbb{T}))).
\end{split}
\end{equation}
which will be obtained by applying the same procedure used to deduce \eqref{eq:s2}. For completeness we highlight the main elements of the proof.
\\
\\
The total $H^{s+1}(\mathbb{T}) \times H^{s}(\mathbb{T})$ balance for $s=3$ reads
\begin{equation}
\begin{split}
& \dfrac{1}{2} d \bigg( \| \partial^3_{x} A \|^2_{L^2} + \| \partial^3_{x} u \|^2_{L^2}  \bigg) +  \int_{\mathbb{T}} \rho^{\alpha-1} | \partial^4_{x} u |^2 dxdt = -  \int_{\mathbb{T}}  \partial^4_x ( A u )\partial^3_{x} Adxdt \\ & - \int_{\mathbb{T}} \partial^4_x( \mu'_k(\rho) \partial_x u ) \partial^3_{x} A dxdt -\int_{\mathbb{T}} \partial^3_{x}(u \partial_x u) \partial^3_{x} u dxdt  -\int_{\mathbb{T}} \partial^3_{x} \bigg( \dfrac{\partial_x p(\rho) }{\rho} \bigg) \partial^3_{x} u dxdt \\ & + \int_{\mathbb{T}} \partial^3_{x} \bigg( \dfrac{\partial_x \mu(\rho) \partial_x u}{\rho} \bigg) \partial^3_{x} u dxdt - \int_{\mathbb{T}}  \partial_x \bigg( \frac{\mu(\rho)}{\rho} \bigg) \partial^3_x u \partial^4_x u dxdt  - \int_{\mathbb{T}} \partial_{xx} \bigg( \frac{\mu(\rho)}{\rho} \bigg) \partial_{xx} u \partial^4_x u dxdt \\ & + \int_{\mathbb{T}} \partial^4_x ( \mu'_k(\rho) \partial_x A) \partial^3_{x} u dxdt + \int_{\mathbb{T}} \partial^3_{x} (A \partial_x A ) \partial^3_{x} u dxdt + \int_{\mathbb{T}} \partial^3_{x} \mathbb{F}(\rho,u) \partial^3_{x} u \, dxdW \\ & + \dfrac{1}{2} \sum_{k=1}^{\infty} \int_{\mathbb{T}} | \partial^3_{x}F_k |^2 \, dxdt
\end{split}
\end{equation}
and we point out the following high order cancellation which is obtained after a tedious but straightforward computation
\begin{equation}\label{sum with cancellation s=3}
\begin{split}
& - \int_{\mathbb{T}} \partial^4_x( \mu'_k(\rho) \partial_x u ) \partial^3_{x} A dx+ \int_{\mathbb{T}} \partial^4_x ( \mu'_k(\rho) \partial_x A) \partial^3_{x} u dx \\ & = -\int_{\mathbb{T}} \bigg( \frac{\beta+1}{2} \bigg) | \partial^3_x A|^2 \partial_x u dx- \int_{\mathbb{T}} (2\beta+2) \partial_{xx} A \partial_{xx} u \partial^3_x A dx \\ & -\int_{\mathbb{T}} (2 \beta+2) A \partial^4_x u \partial^3_x Adx - \int_{\mathbb{T}} (4\beta+4) \partial_x A \partial_{xx} A  \partial^4_x u dx
\end{split}
\end{equation}
thus we have 
\begin{align}
& \dfrac{1}{2} d \bigg( \| \partial^3_{x} A \|^2_{L^2} + \| \partial^3_{x} u \|^2_{L^2}  \bigg) +  \int_{\mathbb{T}} \rho^{\alpha-1} | \partial^4_{x} u |^2 dxdt = -  \int_{\mathbb{T}}  \partial^4_x ( A u )\partial^3_{x} Adxdt \nonumber\\ & -\int_{\mathbb{T}} \partial^3_{x}(u \partial_x u) \partial^3_{x} u dxdt  -\int_{\mathbb{T}} \partial^3_{x} \bigg( \dfrac{\partial_x p(\rho) }{\rho} \bigg) \partial^3_{x} u dxdt + \int_{\mathbb{T}} \partial^3_{x} \bigg( \dfrac{\partial_x \mu(\rho) \partial_x u}{\rho} \bigg) \partial^3_{x} u dxdt \nonumber\\ & - \int_{\mathbb{T}}  \partial_x \bigg( \frac{\mu(\rho)}{\rho} \bigg) \partial^3_x u \partial^4_x u dxdt  - \int_{\mathbb{T}} \partial_{xx} \bigg( \frac{\mu(\rho)}{\rho} \bigg) \partial_{xx} u \partial^4_x u dxdt -\int_{\mathbb{T}} \bigg( \frac{\beta+1}{2} \bigg) | \partial^3_x A|^2 \partial_x u dx\label{s=3}\\ & - \int_{\mathbb{T}} (2\beta+2) \partial_{xx} A \partial_{xx} u \partial^3_x A dx -\int_{\mathbb{T}} (2 \beta+2) A \partial^4_x u \partial^3_x Adx - \int_{\mathbb{T}} (4\beta+4) \partial_x A \partial_{xx} A  \partial^4_x u dx\nonumber\\ & + \int_{\mathbb{T}} \partial^3_{x} (A \partial_x A ) \partial^3_{x} u dxdt + \int_{\mathbb{T}} \partial^3_{x} \mathbb{F}(\rho,u) \partial^3_{x} u \, dxdW + \dfrac{1}{2} \sum_{k=1}^{\infty} \int_{\mathbb{T}} | \partial^3_{x}F_k |^2 \, dxdt=\sum_{k=1}^{13} |I_i| \nonumber
\end{align}
Again, we estimate the right hand side of \eqref{s=3} and then we apply a stochastic Gronwall argument.
\begin{equation}\label{I_1 s=3}
\begin{split}
|I_1| & \le \| \partial^3_x A\|^2_{L^2} \| \partial_x u \|_{L^\infty} + \| \partial^3_x A \|^2_{L^2} \| \partial_{xx} u \|^2_{L^2} + \| \partial_{xx} A\|^2_{L^2}+ \| \partial^3_x u \|^2_{L^2} \| \partial_x A \|^2_{L^\infty} \\ & + C(\delta) \| \partial^3_x A \|^2_{L^2} \| A \|_{L^\infty} \| \rho^{m_1}\|^2_{L^\infty} + \delta \| \rho^{\frac{\alpha-1}{2}} \partial^4_x u \|^2_{L^2},
\end{split}
\end{equation}
\begin{equation}
|I_2| \le \| \partial^3_x u \|^2_{L^2} \| \partial_{xx} u \|_{L^2} + C(\delta) \| \partial^3_x u \|^2_{L^2} \| u \|^2_{L^\infty} \| \rho^{m_1} \|^2_{L^\infty} + \delta \| \rho^{\frac{\alpha-1}{2}} \partial^4_x u \|^2_{L^2},
\end{equation}
\begin{equation}
|I_3| \le  \| \partial^3_{x} u \|^2_{L^2}+ \| \partial_{xx} A \|^2_{L^2} \| \rho^{m_1} \|^2_{L^\infty}+ \| \partial^3_{x} A \|^2_{L^2} \| \rho^{m_1} \|^2_{L^\infty}.
\end{equation}
Similarly we use H\"{o}lder, Sobolev, Poincarè and Young inequalities to estimate the viscosity terms
\begin{equation}
\begin{split}
|I_4|+ |I_5|+ |I_6| & \le \| \partial^3_{x} A \|^2_{L^2} \| \rho^{m_1} \|^2_{L^\infty} + \| \partial^3_{x} u \|^2_{L^2} \| \partial_x u \|^2_{L^\infty}+ \| \partial^3_x u \|^2_{L^2} \| \partial_{xx} A \|^2_{L^2} \\ & + \| \partial_{xx} u \|^2_{L^2} \| \rho^{m_1} \|^2_{L^\infty} + C(\delta) \| \rho^{m_1}\|^2_{L^\infty} \| \partial_x A \|^2_{L^2} \| \partial^3_x u \|^2_{L^2} + \delta \| \rho^{\frac{\alpha-1}{2}} \partial^3_x u \|^2_{L^2},
\end{split}
\end{equation}
while for the terms arising from the capillarity tensor we have 
\begin{equation}
\begin{split}
|I_7|+|I_8|+|I_9|+|I_{10}| + |I_{11}| & \le \| \partial^3_x A \|^2_{L^2} \| \partial_{xx} u \|^2_{L^2} + \| \partial_{xx} A \|^2_{L^2} \\ & + C(\delta) \| A \|_{L^\infty} \| \partial^3_x A \|^2_{L^2} \| \rho^{m_1} \|^2_{L^\infty} + \delta \| \rho^{\frac{\alpha-1}{2}} \partial^4_x u \|^2_{L^2} \\ & + C(\delta) \| \partial_x A \|^2_{L^\infty} \| \partial_{xx} A \|^2_{L^2} \| \rho^{m_1} \|^2_{L^\infty} + \| \partial^3_x A \|^2_{L^2} \| \partial_x A\|_{L^\infty}.
\end{split}
\end{equation}
The It\^{o} reminder term $I_9$ is estimated by using \eqref{F1}-\eqref{F2} as follows
\begin{equation}\label{I_9 s=3}
\begin{split}
|I_{13}| &=\frac{1}{2} \sum_{k=1}^{\infty} \int_{\mathbb{T}} |\partial^3_{x} F_k|^2 dx=\frac{1}{2} \sum_{k=1}^{\infty} \| \partial^3_{x} F_k \|^2_{L^2} \le c \sum_{k=1}^{\infty} \| \partial^3_{x} F_k \|^2_{L^\infty} \le c \sum_{k=1}^{\infty} {\alpha_k}^2 \le C.
\end{split}
\end{equation}
Summing up the estimates \eqref{I_1 s=3}-\eqref{I_9 s=3} and by choosing  the Young constant $\delta$ small enough we get
\begin{equation}
\begin{split}
& \dfrac{1}{2}\text{d} ( \| \partial^3_{x} A \|^2_{L^2} + \| \partial^3_{x} u \|^2_{L^2} ) + c\int_{\mathbb{T}} \rho^{\alpha-1} | \partial^4_x u |^2 dx \\ &  \le \bigg(\| \partial^3_{x} A \|^2_{L^2} + \| \partial^3_{x} u \|^2_{L^2}\bigg)a^{(3)}(t)+b^{(3)}(t)+I_{12},
\end{split}
\end{equation}
with $c>0$ and 
\begin{equation*}
\begin{split}
& a^{(3)}(t)= \| \partial_x u \|_{L^\infty}+ \| \partial_{xx} u \|^2_{L^2} + \| \partial_x A \|^2_{L^\infty} + \| A \|_{L^\infty} \| \rho^{m_1} \|^2_{L^\infty} + \| u\|^2_{L^\infty} \| \rho^{m_1} \|^2_{L^\infty} , \\ & 
b^{(3)}(t)= \| \partial_{xx} A \|^2_{L^2} \| \rho^{m_1} \|^2_{L^\infty} + \|\partial_{xx} u \|^2_{L^2} \| \rho^{m_1} \|^2_{L^\infty} +C.
\end{split}
\end{equation*}
Now since also in this case all the terms appearing in $a^{(3)}$ are controlled in $L^1(\Omega; L^1(0,T)),$ we recall \eqref{gamma2} and we define the stopping time
\begin{equation}\label{gamma3}
\gamma^{(3)}_M= \gamma^{(2)}_M \land \inf \big\{ t\ge 0 \; :  \int_{0}^{t \land \tau} a^{(3)}(s)ds > M \big\}.
\end{equation}
Therefore, we argue as in the previous cases and we use the auxiliary stopping times $\tau_a$ and $\tau_b$ satisfying $\tau_a < t < \tau_b < T \land \gamma^{(3)}_M,$ then we estimate the stochastic integral $I_{12}$ by virtue of the Burkholder-Davis-Gundy inequality, \eqref{F1}-\eqref{F2} and similarly to \eqref{befor gronw gamma2} we deduce the following inequality.
\begin{equation}
\begin{split}
& \mathbb{E} \bigg[ \sup_{ \tau_a \le s \le \tau_b} \frac{1}{2}( \| \partial^3_{x} A \|^2_{L^2} + \| \partial^3_{x} u \|^2_{L^2} )\bigg] + c\mathbb{E} \bigg[ \int_{\tau_a}^{\tau_b}\int_{\mathbb{T}} \rho^{\alpha-1} | \partial^4_x u |^2 dxdt \bigg] \\ & \le \mathbb{E} ( \| \partial^3_{x} A(\rho(\tau_a)) \|^2_{L^2} + \| \partial^3_{x} u (\tau_a) \|^2_{L^2} )+ \mathbb{E} \int_{\tau_a}^{\tau_b}(\| \partial^3_{x} A \|^2_{L^2} + \| \partial^3_{x} u \|^2_{L^2})a^{(3)}(t)dt \\ & +\mathbb{E} \int_{\tau_a}^{\tau_b}b^{(3)}(t)dt,
\end{split}
\end{equation}
from which we apply the stochastic Gronwall Lemma and observe that $\lim_{M \to \infty} \gamma^{(3)}_M= \infty$ a.s.  to infer \eqref{eq:s3}. This, together with the estimate for the lower order derivative norms, concludes the proof of Proposition \ref{Prop global s+1 s}.
For completeness, we also observe that by defining
\begin{equation}\label{min gamma M}
\gamma_M= \min_{s=1,2,3} \gamma^{(s)}_M,
\end{equation}
then trivially $\gamma_M \rightarrow \infty$ a.s as $M \rightarrow \infty.$ 
\end{proof}
\section{Local well-posedness}\label{Sec4}
\noindent
This section is devoted the the proof of the local well-posedness result, Theorem \ref{Main Theorem local}.  In particular, we infer existence and uniqueness of a maximal strong local pathwise solution $(\rho,u, \tau),$ where $\tau$ is a strictly positive stopping time depending on the $W^{2,\infty}$ norm of the density and the velocity.
We consider a multi-layer approximation scheme in the spirit of \cite{Breit Feir Hof 3}, \cite{Feir}. Then, we solve each level of the approximation scheme and we prove convergence results by means of stochastic compactness arguments. 
For the reader convenience, we summarise here the main strategy of the proofs for each level of approximation.
\begin{itemize}
\item At first level of approximation we introduce a class of cut off operators $\theta_R(y)$ which are applied to the nonlinearities of system \eqref{main system}.
\item In order to solve the truncated system we consider the Galerkin projection over the finite dimensional space of trigonometric polynomial of order $m$ and then we perform the limit as $m \rightarrow \infty. $ To solve the Galerkin system we apply a standard fixed point argument. 
\item The limit $m\rightarrow \infty$ is obtained by combining Prokhorov's theorem and Skorokhod's theorem. In particular, we prove the tightness of the family of the laws of solutions over a suitable path space and we identify the limit as a solution of the truncated system.  Note however that at this point only existence of a strong martingale solutions is recovered and hence an additional argument is required. To this purpose we prove a pathwise uniqueness result and we make use of the Gyongy-Krylov characterization of the convergence in probability to return to the original system.
\item The solvability of the second layer of the approximation scheme relies on the introduction of a suitable stopping time $\tau_R$ which depend on the $W_x^{2, \infty}$ norm of the solution, together with the use of an appropriate change of variable and stopping time arguments.
\end{itemize}
We emphasize that in the deterministic setting, local well‑posedness is typically established through high‑order regularity estimates, which also involve time derivatives. The stochastic framework is therefore more restrictive due to the limited regularity in-time of the stochastic elements.  Our analysis is based on the technique developed in \cite{Breit Feir Hof 3} and in the monograph \cite{Feir}, see also \cite{D.P.S.} for the analysis of the Quantum-Navier-Stokes equations and to \cite{Cho 2004}, \cite{Don-Pes} for local well-posedness results and related blow-up criteria for deterministic models of compressible fluids.
\\
\\
We start by introducing the new variables $(r,u)$ 
\begin{equation}\label{new variables}
r:= \begin{cases} 
\big( \frac{2}{\beta+1} \big) \rho^{\frac{\beta+1}{2}} \; \quad \beta \neq -1, \\
\log \rho \quad \quad \quad \quad \beta=-1.
\end{cases}
\end{equation}
Note that with the following choice of $r,$ we recover the augmented formulation \eqref{eq A(rho)}-\eqref{u systm} used in Section 3. The following relations indeed hold $$\partial_x r(\rho)= A(\rho), \quad \; \bigg( \frac{\beta+1}{2}\bigg) r(\rho)= \mu'_k(\rho).$$ Therefore, the skew-adjoint structure of the equation yields several cancellations  also in the estimates for the local existence result.  We rewrite system \eqref{main system} as follows
\begin{equation}\label{local eq r}
\partial_t r+ u \partial_x r+ \mu'_k(\rho) \partial_x u= 0
\end{equation}
\begin{equation}\label{local eq u}
du+ \big[u\partial_x u+ \dfrac{\partial_x p(\rho)}{\rho}\big]dt= \bigg[\dfrac{\partial_x (\mu(\rho)\partial_x u)}{\rho}+ \partial_x( \mu'_k(\rho) \partial_{xx} r) + \partial_x r \partial_{xx} r\bigg]dt+ \mathbb{F}(\rho,u) dW.
\end{equation}
with initial conditions $(r_0,u_0)$
\begin{equation} \label{Approx C.I}
r_0:= \begin{cases} 
\big( \frac{2}{\beta+1} \big) \rho_0^{\frac{\beta+1}{2}} \; \quad \beta \neq -1, \\
\log \rho_0 \quad \quad \quad \; \; \beta=-1.
\end{cases}
\end{equation}
Note that, in order to avoid the use of heavy notations, we keep the dependence of the functions $p, \, \mu, \, \mu'_k$ and $\mathbb{F}$ on the old variable $\rho$ instead of making explicit their dependence on $r.$
\subsection{First level of approximation}
We introduce the fist level of the approximation scheme. In particular we define the family of cut-off operators $\theta_R: [0,\infty) \rightarrow [0,1]$ which are smooth non-increasing functions satisfying 
\begin{equation*}
\theta_R(y)=\begin{cases} 1, \quad 0 \le y \le R \\
0, \quad y \ge R+1.
\end{cases}
\end{equation*}
and we consider the following approximating system 
\begin{equation}\label{trunc local eq r}
\partial_t r+\theta_R(y) (u \partial_x r)+ \theta_R(y)(\mu'_k(\rho) \partial_x u)= 0
\end{equation}
\begin{equation}\label{trunc local eq u}
\begin{split}
du+ \theta_R(y)\big[u\partial_x u+ \dfrac{\partial_x p(\rho)}{\rho}\big]dt & =  \theta_R(y)\bigg[\dfrac{\partial_x (\mu(\rho)\partial_x u)}{\rho}+ \partial_x( \mu'_k(\rho) \partial_{xx} r) + \partial_x r \partial_{xx} r\bigg]dt \\ & + \theta_R(y)\mathbb{F}(\rho,u) dW.
\end{split}
\end{equation}
and our choice for the argument of the cut-off operators is $y= \| r \|_{W^{2,\infty}_x}+ \| u \|_{W^{2,\infty}_x}.$
\\
\\
We use the cut-off operators $\theta_R$ in order to deal with the nonlinearities of system \eqref{local eq r}-\eqref{local eq u} and the truncated terms turn out to be globally Lipschitz continuous functions.  Clearly,  for $\beta=-1$ we have that $\mu'_k(\rho)=1$ and $r=\log \rho.$ In this case the terms depending on $\mu'_k(\rho)$ in \eqref{trunc local eq r} and \eqref{trunc local eq u} are linear and thus no truncations are required. We highlight that this case has been already addressed in \cite{D.P.S.} by the author, Donatelli and Spirito.
Furthermore, we choose the argument of $\theta_R(y)$ to be the $W^{2,\infty}$ norm of the solution, hence condition \eqref{limit cond} in Definition \ref{Def2} follows by construction.  This will be highlighted in the stopping time argument.  \\
\\
Now we state different definitions of solutions for the approximating system \eqref{trunc local eq r}-\eqref{trunc local eq u}.
We start by introducing the class of strong martingale solutions. These solutions are strong in PDEs and weak in probability sense.  Specifically, they satisfy the time integrated version of \eqref{trunc local eq r}-\eqref{trunc local eq u} while the Wiener process $W$ and the stochastic basis with right continuous filtration $(\Omega, \mathcal{F},(\mathcal{F}_t)_{t \ge 0},\mathbb{P})$ are not a priori given but they are considered as part of the solution.
\begin{definition} (Strong martingale solution).  \label{def strong mart}
Let $\Lambda$ be a Borel probability measure on $H^{s+1}(\mathbb{T}) \times H^s(\mathbb{T}), \; s \in \mathbb{N}.$ A multiplet $$((\Omega, \mathcal{F}, (\mathcal{F}_t)_{t \ge 0}, \mathbb{P}),r,u, W)$$ is called a strong martingale solution to the approximated system \eqref{trunc local eq r}-\eqref{trunc local eq u} with initial law $\Lambda$,  provided that the following conditions hold:
\begin{itemize}
\item[(1)]
$(\Omega, \mathcal{F},(\mathcal{F}_t)_{t \ge 0},\mathbb{P}), $ is a stochastic basis with a complete right-continuous filtration;
\item[(2)]
$W$ is an $(\mathcal{F}_t)$-cylindrical Wiener process;
\item[(3)]
$r$ is a $H^{s+1}(\mathbb{T})$-valued $(\mathcal{F}_t)$-progressively measurable stochastic process satisfying
 $$ r \in L^2(\Omega; C([0,T];H^{s+1}(\mathbb{T})))\quad r >0, \; \  \mathbb{P}-\text{a.s.}; $$
 \item[(4)]
 the velocity $u$ is a $H^s(\mathbb{T})$-valued $(\mathcal{F}_t)$-progressively measurable stochastic process satisfying
 $$ u \in L^2(\Omega; C([0,T];H^s(\mathbb{T}))) \cap L^2(\Omega; L^2(0,T;H^{s+1}(\mathbb{T})));$$
 \item[(5)]
 there exists an $\mathcal{F}_0$-measurable random variable $[r_0,u_0]$ such that $\Lambda= \mathcal{L}[r_0,u_0];$
\item[(6)]
the equations
$$
r(t)=r_0- \int_{0}^{t}\theta_R(y)[ u\partial_x r] ds-\int_{0}^{t}\theta_R(y)[ \mu'_k(\rho)\partial_x u]\text{d}s$$
\begin{equation*}
\begin{split}
u(t)&=u_0-\int_{0}^{t}\theta_R(y)[u\partial_x u]ds - \int_{0}^{t}\theta_R(y)\bigg[\frac{\partial_x p(\rho)}{\rho}\bigg]ds +\int_{0}^{t}\theta_R(y)\bigg[ \frac{\partial_x(\mu(\rho)\partial_x u}{\rho} \bigg]ds \\ & +\int_{0}^{t}\theta_R(y)[\partial_x(\mu'_k(\rho) \partial_{xx} r ]ds +\int_{0}^{t}\theta_R(y)[ \partial_x r \partial_{xx}r]ds +\int_{0}^{t} \theta_R(y) \mathbb{F}(\rho,u)\text{d}W
\end{split}
\end{equation*}
holds for all $t \in [0,T] \ \mathbb{P}-a.s.;$
\end{itemize}
\end{definition}
\noindent
For completeness, by analogy with Definition \ref{Def1}, we state the definition of strong pathwise solution for the approximating systems.  These solutions are strong also in the probability sense, meaning that the stochastic setting is given a priori.
\begin{definition}(Strong pathwise solution). \label{strong path sol}
Let $(\Omega, \mathcal{F},(\mathcal{F}_t)_{t \ge 0},\mathbb{P}), $ be a stochastic basis with a complete right-continuous filtration and let $W$ be an $(\mathcal{F}_t)$-cylindrical Wiener process.
Then $(r,u)$ is called a strong pathwise solution to the approximate system \eqref{trunc local eq r}-\eqref{trunc local eq u} with initial conditions $(r_0,u_0)$,  provided that the following conditions hold:
\begin{itemize}
\item[(1)]
 $r$ is a $H^{s+1}(\mathbb{T})$-valued $(\mathcal{F}_t)$-progressively measurable stochastic process satisfying
 $$r\in L^2(\Omega; C([0,T];H^{s+1}(\mathbb{T}))), \quad r >0, \;  \  \mathbb{P}-\text{a.s.}; $$
 \item[(2)]
 the velocity $u$ is a $H^s(\mathbb{T})$-valued $(\mathcal{F}_t)$-progressively measurable stochastic process satisfying
 $$ u \in L^2(\Omega; C([0,T];H^s(\mathbb{T}))) \cap L^2(\Omega; L^2(0,T;H^{s+1}(\mathbb{T})));$$
\item[(3)]
the equations
$$
r(t)=r_0- \int_{0}^{t}\theta_R(y)[ u\partial_x r] ds-\int_{0}^{t}\theta_R(y)[ \mu'_k(\rho)\partial_x u]\text{d}s$$
\begin{equation*}
\begin{split}
u(t)&=u_0-\int_{0}^{t}\theta_R(y)[u\partial_x u]ds - \int_{0}^{t}\theta_R(y)\bigg[\frac{\partial_x p(\rho)}{\rho}\bigg]ds +\int_{0}^{t}\theta_R(y)\bigg[ \frac{\partial_x(\mu(\rho)\partial_x u}{\rho} \bigg]ds \\ & +\int_{0}^{t}\theta_R(y)[\partial_x(\mu'_k(\rho) \partial_{xx} r ]ds +\int_{0}^{t}\theta_R(y)[ \partial_x r \partial_{xx}r]ds +\int_{0}^{t} \theta_R(y) \mathbb{F}(\rho,u)\text{d}W
\end{split}
\end{equation*}
holds for all $t \in [0,T] \ \mathbb{P}-a.s.;$
\end{itemize}
\end{definition}
\noindent
The main result we aim to prove for the approximating system is the following
\begin{theorem}\label{well posedness trunc system}
Let the coefficients $F_k$ satisfy the hypothesis \eqref{F1}-\eqref{F2} and let the initial datum $(r_0,u_0)$ having regularity \eqref{C.I. Momenta strong}. 
Then the following conditions hold:
\begin{itemize}
\item[(1)]
There exists a strong martingale solution to the problem \eqref{trunc local eq r}-\eqref{trunc local eq u} in the sense of Definition \ref{def strong mart} with initial law $\Lambda= \mathcal{L}[(r_0,u_0)].$ Moreover,  there exists a deterministic constant $\bar r_R$ such that $$r(\cdot, t) \ge \overline{r}_R \quad \mathbb{P}-\text{a.s.} \ \text{for all} \ t \in [0,T]$$
and 
\begin{equation}
\mathbb{E} \bigg[ \frac{1}{2}\sup_{t \in [0,T]} \bigg( \| r(t)\|^2_{H^{s+1}(\mathbb{T})}+ \| u(t) \|^2_{H^{s}(\mathbb{T})} \bigg) + \int_{0}^{T} \| u \|^2_{H^{s+1}(\mathbb{T})} dt \bigg]^p \le c(R, r_0,u_0, p) < \infty
\end{equation}
for all $1 \le p < \infty$.
\\
\item[(2)]
If $s > \frac{7}{2},$ then pathwise uniqueness holds.  Specifically,  if $(r^1,u^1),  (r^2,u^2)$ are two strong solutions to \eqref{trunc local eq r}-\eqref{trunc local eq u} defined on the same stochastic basis with the same Wiener process $W$ and $$\mathbb{P}(r_0^1=r_0^2,  u_0^1=u_0^2)=1,  $$
then $$\mathbb{P}(r^1(t)=r^2(t),  u^1(t)=u^2(t), \  \text{for all} \ t \in [0,T])=1. $$
\end{itemize}
Consequently,  there exists a unique strong pathwise solution to \eqref{trunc local eq r}-\eqref{trunc local eq u} in the sense of Definition \ref{strong path sol}.
\end{theorem}
\noindent
We highlight that the stochastic compactness argument used in order to deduce the existence of solutions for the approximating system \eqref{trunc local eq r}-\eqref{trunc local eq u} does not keep track of the original stochastic setting and thus it gives rise to solutions which are defined on a new probability space,  namely martingale solutions. On the other hand, by virtue of the pathwise uniqueness result, it is possible to recover the original stochastic basis by using the Gyongy-Krylov method.  This procedure however requires the additional regularity $s>7/2$ which is not needed in the deterministic setting $\mathbb{F}=0.$ 
\subsection{Second level of approximation}
We use a Galerkin approximation technique.  To be precise, let $H_m$ be the space of the trigonometric polynomials of order $m$ $$ H_m= \bigg\{ v= \sum_{m,\max_{j=1,2,3} |m_j| \le m}^{}[ a_m \cos(\pi m \cdot x)+b_m \sin (\pi m\cdot x)] \; | \; a_m, b_m \in \mathbb{R}\bigg\}$$ having $(w_m)_{m\in \mathbb{N}}$ as orthonormal basis, we define the associated $L^2$- orthogonal projection $${\Pi}_m : L^2 (\mathbb{T}) \longrightarrow H_m.$$
providing the following standard conditions
\begin{equation}
\| {\Pi}_m [f] \|_{W^{k,p}} \le c(k,p) \| f \|_{W^{k,p}},
\end{equation}
\begin{equation}
{\Pi}_m [f] \longrightarrow f \ \text{in} \ W^{k,p}(\mathbb{T}) \ \text{as} \; m \rightarrow \infty, \; \text{for all} \ 1< p < \infty,
\end{equation}
whenever $ f\in W^{k,p}.$
\\
\\
Then we set $r_m={\Pi}_m [r], \, u_m={\Pi}_m [u]$ and we define the projected system of \eqref{trunc local eq r}-\eqref{trunc local eq u}  on $H_m$ as follows 
\begin{equation} \label{Gal r}
\text{d}\langle r_m, w_i \rangle+ \theta_R(y_m)\langle u_m\partial_x r_m,w_i \rangle dt+ \theta_R(y_m)\langle \mu'_k(\rho)\partial_x u_m,w_i \rangle dt=0 
\end{equation}
\begin{equation} \label{Gal u}
\begin{split}
& \text{d}\langle u_m,w_i \rangle+ \theta_R(y_m)\big[\langle u_m\partial_x u_m,w_i \rangle + \big\langle \dfrac{\partial_x p(\rho_m)}{\rho_m}, w_i\big\rangle \big]dt =  \theta_R(y_m)\bigg[\big\langle\dfrac{\partial_x (\mu(\rho_m)\partial_x u_m)}{\rho_m}, w_i\big\rangle\bigg]dt \\ & + \theta_R(y_m) \big[ \langle \partial_x( \mu'_k(\rho_m) \partial_{xx} r_m),w_i \rangle  + \langle \partial_x r_m \partial_{xx} r_m, w_i \rangle \big]dt + \theta_R(y_m)\langle \mathbb{F}(\rho_m,u_m),w_i \rangle dW.
\end{split}
\end{equation}
with initial datum $(r_m(0),u_m(0))=(\Pi_m r_0, \Pi_m u_0).$
\\
\\
The solvability of the first level of the approximation scheme \eqref{Gal r}-\eqref{Gal u} is standard and can be obtained by by means of a Banach fixed point argument.  
For completeness we define $$\mathfrak{B}= L^2(\Omega; C([0,T^*];H_m)), \quad 
\mathcal{F}: \mathfrak{B}\times\mathfrak{B} \longrightarrow \mathfrak{B}\times \mathfrak{B}
$$
with $\mathcal{F}(r_m,u_m)=(\mathcal{F}_1(r_m,u_m), \mathcal{F}_2(r_m,u_m))$ satisfying 
\begin{equation}\label{F1 trunc}
\begin{split}
\langle \mathcal{F}_1(r_m,u_m)(\tau), w_i \rangle &= \langle r_m(0),w_i \rangle-\int_{0}^{\tau} \theta_R(y_m)\langle u_m\partial_x r_m,w_i \rangle dt \\ & -\int_{0}^{\tau} \theta_R(y_m)\langle \mu'_k(\rho)\partial_x u_m,w_i \rangle dt,
\end{split}
\end{equation}
\begin{equation}\label{F2 trunc}
\begin{split}
& \langle \mathcal{F}_2(r_m,u_m)(\tau),  w_i \rangle= \langle u_m(0),w_i \rangle-\int_{0}^{\tau} \theta_R(y_m) \langle u_m\partial_x u_m,w_i \rangle dt -\int_{0}^{\tau} \theta_R(y_m) \big\langle \dfrac{\partial_x p(\rho_m)}{\rho_m}, w_i\big\rangle dt \\ & + \int_{0}^{\tau}\theta_R(y_m) \big\langle \dfrac{\partial_x (\mu(\rho_m)\partial_x u_m)}{\rho_m}, w_i\big\rangle dt + \int_{0}^{\tau} \theta_R(y_m)  \langle \partial_x( \mu'_k(\rho_m) \partial_{xx} r_m),w_i \rangle dt \\ & +\int_{0}^{\tau} \theta_R(y_m)\langle \partial_x r_m \partial_{xx} r_m, w_i \rangle dt + \int_{0}^{\tau} \theta_R(y_m)\langle \mathbb{F}(\rho_m,u_m),w_i \rangle dW.
\end{split}
\end{equation}
Then,  for $(r^1_m,u^1_m)$ and $(r^2_m,u^2_m)$ being two different solutions of \eqref{Gal r}-\eqref{Gal u} and by denoting by $\mathcal{F}_{det}$ and $\mathcal{F}_{sto}$ the deterministic and stochastic integrals respectively,  since all norms defined on the finite dimensional space $H_m$ are equivalent, we easily deduce $$\| \mathcal{F}(r^1_m,u^1_m)-\mathcal{F}(r^2_m,u^2_m) \|^2_B \le \|\mathcal{F}_{det}(r^1_m,u^1_m)-\mathcal{F}_{det}(r^2_m,u^2_m)\|^2_B+ \|\mathcal{F}_{sto}(r^1_m,u^1_m)-\mathcal{F}_{sto}(r^2_m,u^2_m))\|^2_B.$$
and 
\begin{equation}
\| \mathcal{F}_{det}(r^1_m,u^1_m)-\mathcal{F}_{det}(r^2_m,u^2_m) \|^2_B \le T^*C(m,R,T) \| (r^1_m,u^1_m)-(r^2_m,u^2_m)\|^2_B,
\end{equation}
while concerning the stochastic part, with similar lines of argument with respect to the previous section, we us Burkholder-Davis-Gundy's inequality to deduce
\begin{equation}
\begin{split}
 \| \mathcal{F}_{sto}(r^1_m,u^1_m)-\mathcal{F}_{sto}(r^2_m,u^2_m)\|^2_B & = \mathbb{E} \ \sup_{[0,T^*]} \bigg{ \| } \int_{0}^{t} (\theta_R( y^1_m) \mathbb{F}(\rho^1_m,u^1_m)-(\theta_R(y^2_m) \mathbb{F}(\rho^2_m,u^2_m) \text{d}W \bigg{ \| }^2_{H_m}  \\ & \le T^*C(m,R,T) \| (r^1_m,u^1_m)-(r^2_m,u^2_m)\|^2_B.
\end{split}
\end{equation}
We choose $C(m,R,T)$ in such a way $T^*C(m,R,T)< 1.$ Therefore $\mathcal{F}$ is a contraction for a deterministic small time $T^*>0$  and by using Banach fixed point Theorem we get the existence of a unique solution in the interval $[0,T^*].$ To infer the existence of a solution in $[0,T]$ we provide a decomposition into small subintervals and we patch the corresponding solutions together.
\\
\\
In order to study the limit $ m\rightarrow \infty$ we need uniform estimates with respect to the parameter $m.$
To this purpose, let $l$ be a multi-index such that $|l| \le s,$ we apply $\partial^{l+1}_x$ to \eqref{Gal r} and then we test with $\partial_x^{l+1}r_m,$
while concerning \eqref{Gal u} we apply It\^{o} formula and by recalling the invariance of $H_m$ with respect to the spatial derivative we deduce 
\begin{equation}\label{Uniform est m}
\begin{split}
& \dfrac{1}{2}\text{d} \int_{\mathbb{T}} \bigg( | \partial_x^{l+1} r_m|^2+| \partial_x^l u_m |^2 \bigg) dx= -\theta_R(y_m)\int_{\mathbb{T}}[ \partial_x^{l+1}(u_m\partial_x r_m)\partial_x^{l+1}r_m dxdt \\ & -\theta_R(y_m) \int_{\mathbb{T}} \partial_x^l(u_m \partial_x u_m)\partial_x^l u_m dxdt - \theta_R(y_m) \int_{\mathbb{T}}\partial_x^l\bigg( \dfrac{\partial_x p(\rho_m)}{\rho_m} \bigg)\partial_x^l u_m dxdt \\ & + \theta_R(y_m) \int_{\mathbb{T}} \partial_x^l \bigg( \dfrac{\partial_x (\mu(\rho_m)\partial_x u_m)}{\rho_m} \bigg)\partial_x^l u_m  dt + \theta_R(y_m)\int_{\mathbb{T}} \partial_x^{l+1}( \mu'_k(\rho_m) \partial_{xx} r_m) \partial_x^l u_m dxdt \\ & + \theta_R(y_m) \int_{\mathbb{T}} \partial_x^l (\partial_x r_m \partial_{xx} r_m) \partial_x^l u_m dxdt + \theta_R(y_m) \int_{\mathbb{T}} \partial_x^l \mathbb{F}(\rho_m,u_m) \partial_x^l u_m dxdW\\ & +\dfrac{1}{2} \sum_{k=1}^{\infty} \int_{\mathbb{T}} \theta_R( y_m) | \partial_x^l F_k(\rho_m,u_m) |^2 dxdt
\end{split}
\end{equation}
The estimates of the right hand-side of \eqref{Uniform est m} follows the same lines of argument with respect to the proof of the high order derivative estimates obtained in Section \ref{Sec3}. Note that by virtue of the equivalence of the norm in the finite dimensional space $H_m$ and the cut-off operator $\theta_R(y_m),$ here the scenario is significantly easier. The proof of the following estimate is therefore omitted
\begin{equation}\label{reg}
\begin{split}
& \mathbb{E} \bigg[ \frac{1}{2}\sup_{[0,T]} \bigg( \| r_m \|^2_{H^{s+1}}+ \| u_m\|^2_{H^{s}} \bigg)+ \int_{0}^{T} \theta_R(y_m) \int_{\mathbb{T}} \rho^{\alpha-1}| \partial_x^{s+1}u_m|^2 dxdt \bigg]^p  \\ & \le c(R,T,s) \mathbb{E} \bigg[ \| r_0 \|_{H^{s+1}}^{2p}+ \| u_0 \|_{H^s}^{2p}+1 \bigg].
\end{split}
\end{equation}
\subsection{Stochastic compactness}
The aim of this subsection is to perform the limit $m \rightarrow \infty.$
\\
We start by defining the following path space 
$\mathcal{X}=\mathcal{X}_r\times\mathcal{X}_u\times\mathcal{X}_W$
where $$\mathcal{X}_r=C([0,T];H^{s_1}(\mathbb{T}), \ \mathcal{X}_u=C([0,T];H^{s_2}(\mathbb{T}),\ \mathcal{X}_W=C([0,T];\mathfrak{U}_0) $$
with $s_1 < s+1$ such that $s_1 > \frac{7}{2}$ and $s_2< s$ such that $s_2> \frac{5}{2}.$ Note that the conditions $s_i > \frac{5}{2}= \frac{n}{2}+2$ is needed in order to pass to the limit in the cut-off operator $\theta_R.$
Then we denote by $\mathcal{L}[r_m],\ \mathcal{L}[u_m],\ \mathcal{L}[W]$ the law of $r_m, u_m,  W$ respectively and we infer the following Lemma
\begin{lemma}\label{Tightness}
The following results hold
\begin{itemize}
\item[1)]
The set $\{ \mathcal{L}[r_m]; \ m\in \mathbb{N} \}$ is tight on $\mathcal{X}_\psi$
\item[2)]
The set $\{ \mathcal{L}[u_m]; \ m\in \mathbb{N} \}$ is tight on $\mathcal{X}_u$
\item[3)]
The set $\{ \mathcal{L}[W] \}$ is tight on $\mathcal{X}_W$
\item [4)]
The joint law $\{\mathcal{L}[r_m,u_m,W]; \ m\in \mathbb{N} \}$ is tight on $\mathcal{X}.$ 
\end{itemize}
\end{lemma}
\begin{proof} 
We start with the proof of the tightness of $\mathcal{L}[u_m].$
\\
\\
For any $L>0$ we define the set 
\begin{equation*}
B_L= \{ u \in C([0,T];H^s(\mathbb{T}) \cap C^{k}([0,T];L^2(\mathbb{T})): \; \| u\|_{C_tH^s_x}+ \| u \|_{C^k_t L^2_x}  \le L \}
\end{equation*}
and we observe that due to the compact embedding 
\begin{equation}\label{emb}
C([0,T];H^s(\mathbb{T}) \cap C^{k}([0,T];L^2(\mathbb{T})) \overset{C}{\hookrightarrow} C([0,T];H^{s_2}(\mathbb{T})), \quad k>0,\; s_2<s,
\end{equation}
which follows directly from Ascoli Arzelà theorem,  then $B_L$ is relatively compact in $\mathcal{X}_u.$ On the other hand by using Markov inequality 
\begin{equation}\label{tight proof}
\mathcal{L}[u_m](B^c_L)= \mathbb{P} \big( \| u_m\|_{C_tH^s_x}+ \| u_m\|_{C^k_t L^2_x} > L \big) \le \frac{1}{L} \mathbb{E} \big( \| u_m\|_{C_tH^s_x}+ \| u_m\|_{C^k_t L^2_x} \big)
\end{equation}
and recalling \eqref{reg} together with the standard estimate 
\begin{equation*}
\mathbb{E} \big( \| u^d_m \|_{C^k_t L^2_x}+\| u^s_m \|_{C^k_t L^2_x} \big) \le C(R)
\end{equation*}
for the decomposition $u_m= u^d_m + u^s_m$ in terms of deterministic integrals and the stochastic integral, see also \cite{Feir}, Proposition 5.2.5 for completeness, we take $L$ large enough in \eqref{tight proof} to deduce our claim.
The proof for $\mathcal{L}[r_m]$ follows the same lines of argument with respect to the proof of the tightness for $\mathcal{L}[u_m]$ by replacing $s_2$ with $s_1.$ Finally,  $\mathcal{L}[W]$ is tight since it is a Radon measure on a Polish space $\mathcal{X}_W$ and the tightness of the joint law follows by Tychonoff theorem.
\end{proof}
\noindent
The next result is a fundamental convergence lemma which follows by a direct application of Prokhorov and Skorokhod Theorem. 
\begin{lemma}\label{conv stoch}
There exists a complete probability space $(\tilde{\Omega}, \tilde{\mathcal{F}},\tilde{\mathbb{P})}, $ with $\mathcal{X}$-valued Borel-measurable random variables $(\tilde{r}_m,\tilde{u}_m,\tilde{W}_m),  m\in \mathbb{N},$ and $(\tilde{r},\tilde{u},\tilde{W}), $ such that up to a subsequence:
\begin{itemize}
\item[(1)]
the law of $(\tilde{r}_m,\tilde{u}_m,\tilde{W}_m), $ is given by $\mathcal{L}[r_m,u_m,W]; \ m\in \mathbb{N};$
\item[(2)] 
the law of $(\tilde{r},\tilde{u},\tilde{W})$ is a Radon measure;
\item[(3)]
$(\tilde{r}_m,\tilde{u}_m,\tilde{W}_m), $ converges $\tilde{\mathbb{P}}-$a.s. to $(\tilde{r},\tilde{u},\tilde{W})$ in the topology of $\mathcal{X},$ namely
\begin{equation*}
\begin{split}
& \tilde{r}_m \rightarrow \tilde{r} \quad  \, in \ C([0,T]; H^{s_1}(\mathbb{T})), \\ & \tilde{u}_m \rightarrow \tilde{u} \quad in \ C([0,T]; H^{s_2}(\mathbb{T})), \\ & 
\tilde{W}_m \rightarrow \tilde{W} \, in \; C([0,T]; \mathfrak{U}_0),
\end{split}
\end{equation*}
\end{itemize}
as $m \rightarrow \infty \ \tilde{\mathbb{P}}-$a.s.
\end{lemma}
\begin{proof}
Since $\mathcal{X}$ is a Polish space and $\{\mathcal{L}[r_m,u_m,W]; \ m\in \mathbb{N} \}$ is tight on $\mathcal{X},$ then by using Prokhorov's theorem,  see \cite{Feir}, Theorem 2.6.1, we have that the collection of measures $\{\mathcal{L}[r_m,u_m,W]\}$ is relatively weak compact. Therefore, the results in Lemma \ref{conv stoch} follows by applying Skorokhod Theorem, Theorem 2.6.2 in \cite{Feir}, which gives a representation of the limit as the law of suitable random variables defined on a new probability space.
\end{proof}
\noindent
The identification of the limit as a strong martingale solution of \eqref{trunc local eq r}-\eqref{trunc local eq u} relies on measurability properties of $(\tilde{\rho}, \tilde{u}, \tilde{W})$ and uniform moment estimates. Specifically, since $\tilde{r}$ and $\tilde{u}$ are stochastic processes in the classical sense, with continuous trajectories,  then with standard arguments we infer they are progressively measurable with respect to the canonical filtration generated by $[ \tilde{r}, \tilde{u}, \tilde{W}]$ i.e. 
\begin{equation}
\tilde{\mathcal{F}}_t:= \sigma \bigg( \sigma_t [\tilde{r}] \cup \sigma_t [\tilde{u}] \cup \bigcup_{k=1}^{\infty} \sigma_t [\tilde{W}_k] \bigg), \quad t \in [0,T].
\end{equation}
\\
Nevertheless, the process $\tilde{W}$ is a cylindrical Wiener process with respect to its canonical filtration and since $\tilde{\mathcal{F}}_t$ is non-anticipative with respect to $\tilde{W}$ then we deduce that $\tilde{W}$ is also an $\tilde{\mathcal{F}}_t$-Cylindrical Wiener process. We refer the reader to \cite{Feir}, Section 5.2.4 for further details.
To conclude, we observe that by virtue of Theorem 2.9.1 in \cite{Feir}, \eqref{Uniform est m}, Proposition \ref{conv stoch} and Lemma 2.6 in \cite{Feir}, the limit $(\tilde{r}, \tilde{u}, \tilde{W})$ solves \eqref{trunc local eq r}-\eqref{trunc local eq u} on the new probability space. This concludes the existence of martingale solutions in Theorem \ref{well posedness trunc system}.
The strong continuity in time of the solution then follows $(\tilde{r}, \tilde{u}),$ by using Gelfand variational approach, see Theorem 2.4.3 in \cite{Feir}, in the Gelfand triplet $H^{s+1}(\mathbb{T}) \hookrightarrow H^{s}(\mathbb{T})\hookrightarrow H^{s-1}(\mathbb{T}).$
\subsection{Pathwise uniqueness}
We derive a pathwise uniqueness result. This condition will be later used in order to apply the Gyongy-Krylov lemma to recover the original probability space and thus to prove that martingale solutions are indeed pathwise solutions.  Here the condition $s> \frac{7}{2}$ will be needed.
\\
\\
Consider $(r_1,u_1)$ and $(r_2,u_2)$ two different solutions of \eqref{trunc local eq r}-\eqref{trunc local eq u}. Then the difference satisfies
\begin{equation}\label{diff1}
\begin{split}
\partial_t(r_1-r_2)+\theta_R(y_1)( u_1\partial_x r_1) -\theta_R(y_2)( u_2\partial_x r_2) = &  -\theta_R(y_1)(\mu'_k(\rho_1)\partial_x u_1) \\ & +\theta_R(y_2)(\mu'_k(\rho_2)\partial_x u_2)
\end{split}
\end{equation}
\begin{equation}\label{diff2}
\begin{split}
 \text{d}(u_1-u_2)=& -\theta_R(y_1)(u_1\partial_x u_1)dt +\theta_R(y_2)(u_2\partial_x u_2)dt -\theta_R(y_1)\bigg( \frac{\partial_x p(\rho_1)}{\rho_1}\bigg)dt \\ & + \theta_R(y_2)\bigg( \frac{\partial_x p(\rho_2)}{\rho_2}\bigg)dt +\theta_R(y_1)\bigg(\frac{\partial_x(\mu(\rho_1)\partial_x u_1)}{\rho_1} \bigg)dt \\ & -\theta_R(y_2)\bigg(\frac{\partial_x(\mu(\rho_2)\partial_x u_2)}{\rho_2} \bigg)dt +\theta_R(y_1)(\partial_x(\mu'_k(\rho_1) \partial_{xx} r_1)dt \\ & -\theta_R(y_2)(\partial_x(\mu'_k(\rho_2) \partial_{xx} r_2)dt +\theta_R(y_1)(\partial_x r_1 \partial_{xx}r_1)dt \\ & -\theta_R(y_2)(\partial_x r_2 \partial_{xx}r_2)dt +\theta_R(y_1) \mathbb{F}(\rho_1,u_1)\text{d}W-\theta_R(y_2) \mathbb{F}(\rho_2,u_2)\text{d}W.
\end{split}
\end{equation}
Therefore, similarly to the proof of Proposition \ref{Prop h2 est} and Proposition \ref{Prop global s+1 s},  after a tedious but straightforward computation we deduce the following balance 
\begin{align*}\label{eq path uniq}
& \frac{1}{2}\text{d}\int_{\mathbb{T}}| \partial _{x}^{l +1}(r _{1}-r_{2})| ^{2}+ | \partial_x^{l}(u_1-u_2)|^2 dx=- \theta_R(y_1)\int_{\mathbb{T}}\partial^{l+1}_x( u_1\partial_x r_1)\partial^{l+1}_x(r_1-r_2)dxdt \\ & +\theta_R(y_2)\int_{\mathbb{T}}\partial^{l+1}_x( u_2\partial_x r_2)\partial^{l+1}_x(r_1-r_2)dxdt -\theta_R(y_1)\int_{\mathbb{T}}\partial^{l+1}_x(\mu'_k(\rho_1)\partial_x u_1)\partial^{l+1}_x(r_1-r_2)dxdt \\ & +\theta_R(y_2)\int_{\mathbb{T}}\partial^{l+1}_x(\mu'_k(\rho_2)\partial_x u_2)\partial^{l+1}_x(r_1-r_2)dxdt -\theta_R(y_1)\int_{\mathbb{T}}\partial^l_x(u_1\partial_x u_1)\partial^l_x(u_1-u_2)dxdt \\ & + \theta_R(y_2)\int_{\mathbb{T}}\partial^l_x(u_2\partial_x u_2)\partial^l_x(u_1-u_2)dxdt-\theta_R(y_1)\int_{\mathbb{T}} \partial^l_x\bigg( \frac{\partial_x p(\rho_1)}{\rho_1}\bigg)\partial^l_x(u_1-u_2)dxdt \\ & + \theta_R(y_2)\int_{\mathbb{T}}\partial^l_x\bigg( \frac{\partial_x p(\rho_2)}{\rho_2}\bigg)\partial^l_x(u_1-u_2)dxdt +\theta_R(y_1)\int_{\mathbb{T}}\partial^l_x\bigg( \frac{\partial_x(\mu(\rho_1)\partial_x u_1)}{\rho_1} \bigg)\partial^l_x(u_1-u_2)dxdt \\ & -\theta_R(y_2)\int_{\mathbb{T}}\partial^l_x\bigg(\frac{\partial_x(\mu(\rho_2)\partial_x u_2)}{\rho_2} \bigg)\partial^l_x(u_1-u_2)dxdt +\theta_R(y_1)\int_{\mathbb{T}}\partial^l_x(\partial_x(\mu'_k(\rho_1) \partial_{xx} r_1)\partial^l_x(u_1-u_2)dxdt \\ & -\theta_R(y_2)\int_{\mathbb{T}}\partial^l_x(\partial_x(\mu'_k(\rho_2) \partial_{xx} r_2)\partial^l_x(u_1-u_2)dxdt +\theta_R(y_1)\int_{\mathbb{T}}\partial^l_x(\partial_x r_1 \partial_{xx}r_1)\partial^l_x(u_1-u_2)dxdt \\ & -\theta_R(y_2)\int_{\mathbb{T}}\partial^l_x(\partial_x r_2 \partial_{xx}r_2)\partial^l_x(u_1-u_2)dxdt +\theta_R(y_1)\int_{\mathbb{T}} \partial^l_x \mathbb{F}(\rho_1,u_1)\partial^l_x(u_1-u_2)dx\text{d}W \\ & -\theta_R(y_2)\int_{\mathbb{T}} \partial^l_x \mathbb{F}(\rho_2,u_2)\partial^l_x(u_1-u_2)dx\text{d}W \\ & + \frac{1}{2} \sum_{k=1}^{\infty} \int_{\mathbb{T}} \big( \theta_R(y_1) \partial^l_x F_k(\rho_1,u_1)-\theta_R(y_2) \partial^l_x F_k(\rho_2,u_2) \big)^2 dxdt,
\end{align*}
which is obtained by applying the operators $\partial_x^{l+1}$ and $\partial_x^{l},$ with $|l| \le s',$ to \eqref{diff1} and to \eqref{diff2} respectively and by using properly the It\^{o} formula to the functional $F(\partial^l_x (u_1-u_2))= \frac{1}{2} \int_{} | \partial^l_x(u_1-u_2)|^2dx.$ 
\\
\\
We highlight the main elements which are needed in order to conclude the pathwise uniqueness estimate result in the sense of Theorem \ref{well posedness trunc system}.
By definition of $\theta_R$ and Sobolev embedding we have that for $s'> \frac{5}{2}$
\begin{equation*}
\begin{split}
| \theta_R(y_1)-\theta_R(y_2) | & \le C_1(R) ( \| r_1-r_2\|_{W_x^{2,\infty}} +\| u_1-u_2 \|_{W_x^{2,\infty}}) \\ & \le C_2(R) ( \| r_1-r_2\|_{H_x^{s'}} +\| u_1-u_2 \|_{H_x^{s'}}).
\end{split}
\end{equation*}
We frequently add and subtract quantities of the form $\theta_R(y_i)$ $\int_{\mathbb{T}}h(r_j,u_j) \partial^{l+1}_x(r_1-r_2)dx, \; i \neq j$ and $h(r_j,u_j)$ properly defined and we have to deal with the following type of commutator estimates
\begin{equation*}
\begin{split}
& \bigg| \theta_R(y_1)\int_{\mathbb{T}}( \partial^{l+1}_x\mu'_k(\rho_1)\partial_x u_1-\partial^{l+1}_x \mu'_k(\rho_2) \partial_x u_2) \partial^{l+1}_x(r_1-r_2)dx \bigg| \\ & =  \bigg| \theta_R(y_1) \int_{\mathbb{T}} \bigg[ \partial_x u_1( \partial^{l+1}_x ( \mu'_k(\rho_1)-\mu'_k(\rho_2) ) -\partial^{l+1}_x \mu'_k(\rho_2) \partial_x(u_1-u_2) \bigg] \partial^{l+1}_x (r_1-r_2) dx \bigg| \\ & \le  \theta_R(y_1) \| u_1\|_{W^{1,\infty}} \| r_1-r_2\|^2_{H^{s'+1}}+  \theta_R(y_1) \| r_2\|_{H^{s'+1}} \| u_1-u_2\|_{W^{1,\infty}} \| r_1-r_2\|_{H^{s'+1}} \\ & \le C(R) ( \| r_1-r_2\|^2_{H^{s'+1}} + \|r_2\|^2_{H^{s'+1}} \| u_1-u_2\|^2_{H^{s'}})
\end{split}
\end{equation*}
\begin{equation*}
\begin{split}
& \bigg| \theta_R(y_1)\int_{\mathbb{T}}( u_1\partial_x \partial^{l+1}_xr_1-u_2 \partial_x \partial^{l+1}_x r_2) \partial^{l+1}_x(r_1-r_2)dx \bigg| \\ & =  \bigg| \theta_R(y_1) \int_{\mathbb{T}} \bigg[ u_1 \partial_x \partial^{l+1}_x (r_1-r_2) + \partial_x \partial^{l+1}_x r_2 (u_1-u_2) \bigg] \partial^{l+1}_x (r_1-r_2) dx \bigg| \\ & \le  \theta_R(y_1) \| u_1\|_{W^{1,\infty}} \| r_1-r_2\|^2_{H^{s'+1}}+  \theta_R(y_1) \| r_2\|_{H^{s'+2}} \| u_1-u_2\|_{L^\infty} \| r_1-r_2\|_{H^{s'+1}} \\ & \le C(R) ( \| r_1-r_2\|^2_{H^{s'+1}} + \|r_2\|^2_{H^{s'+2}} \| u_1-u_2\|^2_{H^{s'}})
\end{split}
\end{equation*}
Similarly to the proof of Proposition \ref{Prop h2 est} and Proposition \ref{Prop global s+1 s}, by virtue of the use of the new variables $(r,u),$ several cancellations between the high-order terms occur.
Additionally, the viscosity terms give rise to some integrals having an appropriate sign and they therefore  provide dissipation properties. Specifically, we have that 
\begin{equation*}
\begin{split}
\theta_R(y_1)\int_{\mathbb{T}}\partial_x^{l} \bigg( \frac{\mu(\rho_1)}{\rho_1} \partial_{xx} u_1 \bigg) \partial^l_x (u_1-u_2)  & =  \theta_R(y_1)\int_{\mathbb{T}}\partial_x^{l} \bigg( \frac{\mu(\rho_1)}{\rho_1} \partial_{xx} (u_1-u_2) \bigg) \partial^l_x (u_1-u_2) \\ & + \theta_R(y_1) \int_{\mathbb{T}} \partial^l_x \bigg( \frac{\mu(\rho_1)}{\rho_1} \partial_{xx} u_2 \bigg) \partial^l_x(u_1-u_2) \\ & = -\theta_R(y_1)\int_{\mathbb{T}} \bigg( \frac{\mu(\rho_1)}{\rho_1}\bigg) | \partial^{l+1}_x (u_1-u_2)|^2 + l.o.t.
\end{split}
\end{equation*}
where the lower order terms can be estimated by standard arguments and the same analysis holds for the terms involving $\theta_R(y_2).$
For completeness, we highlight that, after having added and subtracted appropriate quantities, the high order terms due to capillarity can be estimated as follows
\begin{equation*}
\begin{split}
\theta_R(y_1) \bigg| \int_{\mathbb{T}}(\mu'_k(\rho_1)-\mu'_k(\rho_2)) \partial_x^{l+3} r_2 \partial_x^{l}(u_1-u_2) dx \bigg| & \le \| r_1-r_2\|^2_{H^{s'+1}}+ \| u_1-u_2\|^2_{H^{s'}} \| r_2 \|^2_{H^{s'+2}} \\ & + \delta \| u_1-u_2\|^2_{H^{s'+1}},
\end{split}
\end{equation*}
\begin{equation*}
\begin{split}
& \theta_R(y_1)\bigg| \int_{\mathbb{T}} \partial_x^l(\partial_x r_1 \partial_{xx} r_1) \partial^l_x(u_1-u_2)-\partial_x^l(\partial_x r_2 \partial_{xx} r_2) \partial^l_x(u_1-u_2) dx \bigg| \\ & = \theta_R(y_1)\bigg| \int_{\mathbb{T}} \partial_x^l(\partial_x r_1 \partial_{xx} (r_1-r_2)) \partial^l_x(u_1-u_2)-\partial_x^l(\partial_x (r_1-r_2) \partial_{xx} r_2) \partial^l_x(u_1-u_2) dx \bigg| \\ & \le c(R) \big( \| r_1-r_2\|^2_{H^{s'+1}}+ \| r_2\|^2_{H^{s'+2}} \| u_1-u_2\|^2_{H^{s'}}+ \delta \| u_1-u_2\|^2_{H^{s'+1}} \big).
\end{split}
\end{equation*}
\\
Therefore we end up with the following inequality
\begin{equation*}
\begin{split}
& \text{d}( \| r_1-r_2\|_{H^{s'+1}}^2+ \| u_1-u_2 \|_{H^{s'}}^2 )  \le g(t) ( \| r_1-r_2\|_{H^{s'+1}}^2+ \| u_1-u_2 \|_{H^{s'}}^2 )dt + \mathcal{R}_1\text{d}W-\mathcal{R}_2\text{d}W,
\end{split}
\end{equation*}
where $$g(t)= C(R)( 1+ \| r_1 \|_{H^{s'+2}}^2 + \| r_2 \|_{H^{s'+2}}^2 + \| u_1 \|_{H^{s'+2}}^2+ \| u_2 \|_{H^{s'+2}}^2),$$
for $s' > \frac{5}{2}$ and 
\begin{equation*}
\mathcal{R}_i= \int_{\mathbb{T}}\sum_{ |l | \le s'} \theta(y_i) \partial_x^{l} \mathbb{F}( \rho_i,u_i) \partial_x^{\beta} (u_1-u_2) dx.
\end{equation*}
Since $g \in L^1(0,T),$ we use It\^{o} product rule to deduce 
\begin{equation}\label{exp path uniq}
\begin{split}
& \text{d} [ e^{-\int_{0}^{t} g(\sigma) d\sigma} ( \|r_1-r_2 \|_{H^{s'+1}}^2 + \| u_1-u_2 \|_{H^{s'}}^2)] \le e^{-\int_{0}^{t} g(\sigma) d\sigma}(\mathcal{R}_1\text{d}W-\mathcal{R}_2\text{d}W)
\end{split}
\end{equation}
and integrating on $[0,t]$ and taking the expectation we have that the stochastic integrals vanishes in expectation. Therefore, since 
\begin{equation*}
\mathbb{P}(r_0^1=r_0^2,  u_0^1=u_0^2)=1,  
\end{equation*}
then
\begin{equation}
\mathbb{E} \big[e^{-\int_{0}^{t} g(\sigma) d\sigma}( \| r_1-r_2\|_{H^{s'+1}}^2+ \| u_1-u_2 \|_{H^{s'}}^2 )\big]=0.
\end{equation}
To conclude we observe that $e^{.-\int_{0}^{t} G(\sigma) d\sigma} > 0 \quad \mathbb{P}-a.s$
and the solutions $r_i,u_i$ have continuous trajectories in $H^{s'}(\mathbb{T})\; a.s. ,$ thus pathwise uniqueness in the sense of Theorem \ref{well posedness trunc system} holds.
\subsection{Solvability of the first level of approximation}
We claim the existence of a strong pathwise solution in the sense of Definition \ref{strong path sol}. Specifically, by virtue of the pathwise uniqueness result and the existence of a strong martingale solutions, we make use of the Gyongy-Krylov characterization of convergence in probability, Lemma 2.10.1 in \cite{Feir}, in order to prove the existence of a pathwise solution. 
\\
\\
Let us consider an initial datum satisfying \eqref{C.I STRONG} for $s> \frac{7}{2}.$  Then, for $(r_m, \, u_m),\, (r_n, \, u_n)$ being two solutions of \eqref{Gal r}-\eqref{Gal u},  by using similar arguments with respect to the proof of Lemma \ref{Tightness},  the join law $\mathcal{L}[r_m, \, u_m,\, r_n, \, u_n,W]$ is tight on $\mathcal{X}^{J}=\mathcal{X}_r \times \mathcal{X}_u \times \mathcal{X}_r \times \mathcal{X}_u \times \mathcal{X}_W.$ Therefore,  Skorokhod representation theorem yields the existence of a complete probability space $(\overline{\Omega}, \mathcal{\overline{F}},\overline{\mathbb{P}})$ together with a sequence of random variables, which eventually passing to a subsequence satisfy $$[\hat{r}_{m_k,} \, \hat{u}_{m_k},\, \bar{r}_{n_k}, \, \bar u_{n_k},\tilde{W}_k], \quad k \in \mathbb{N}$$
and $[\hat{r},\, \hat{u},\, \bar{r}
, \, \bar u_,\tilde{W}]$ such that $$[\hat{r}_{m_k,} \, \hat{u}_{m_k},\, \bar{r}_{n_k}, \, \bar{u}_{n_k},\tilde{W}_k] \rightarrow [\hat{r},\, \hat{u},\, \bar{r}, \, \bar{u}_,\tilde{W}] \quad  \text{in} \; \mathcal{X}^{J}, \; \mathbb{\overline{P}} -a.s. $$
and $$\mathcal{L}[\hat{r}_{m_k,} \, \hat{u}_{m_k},\, \bar{r}_{n_k}, \, \bar{u}_{n_k},\tilde{W}_k]=\mathcal{L}[{r}_{m_k},\, {u}_{m_k},\, r_{n_k} \, u_{n_k},{W}] \quad  \text{in} \; \mathcal{X}^{J}$$
Moreover $\mathcal{L}[r_{m_k},\, {u}_{m_k},\, r_{n_k} \, u_{n_k},{W}]$ converges weakly to the measure $\mathcal{L}[\hat{r},\hat{u},  \bar{r}, \bar{u}, \tilde{W}].$ 
Clearly,  both $[\hat{r},\hat{u}, \tilde{W}] $ and $[ \bar{r},\bar{u}, \tilde{W}]$ are strong martingale solutions to the approximate system \eqref{trunc local eq r}-\eqref{trunc local eq u} and since $\mathbb{\bar P}(\hat{r}(0)=\bar{r}(0))=1.$ and 
$\mathbb{\bar P}(\hat{u}(0)=\bar{u}(0))=1,$ we use pathwise uniqueness to deduce 
$$\mathcal{L}[ \hat{r},  \hat{u},  \bar{r}, \bar{u}]([ r_1,u_1, r_2,u_2]; [r_1,u_1]=[r_2,u_2])= \mathbb{ \bar{P}}([ \hat{r}, \hat{u}]=[ \bar{r},\bar{u}])=1$$
from which we infer that the original sequence $[r_m,u_m]$ is defined on the initial probability space $(\Omega, \mathcal{F}, \mathbb{P}),$ and it converges in probability in the topology of $\mathcal{X}_r \times \mathcal{X}_u$ to a random variable $[r,u].$ Passing to a subsequence we assume the convergence is almost surely and thus we identify the limit as the unique strong pathwise solution of the approximate problem \eqref{trunc local eq r}-\eqref{trunc local eq u}.  
\subsection{Proof of Theorem \ref{Main Theorem local}}
To avoid technical issues, we assume the following extra regularity on the initial data
\begin{equation}\label{extra hp initial data}
(\rho_0,u_0) \in L^\infty(\Omega; H^{s+1}(\mathbb{T})) \times L^\infty(\Omega; H^{s}(\mathbb{T}))), \quad \rho_0 \ge C >0.
\end{equation}
for $s \ge \frac{7}{2}.$ Hypothesis \eqref{extra hp initial data} can be then relaxed to \eqref{C.I STRONG} by means of a standard decomposition of the space $\Omega,$ see for instance \cite{Feir} and \cite{D.P.S.}.
Uniqueness of maximal strong pathwise solutions then follows by the pathwise uniqueness result given in Theorem \ref{well posedness trunc system}. In particular, given $(\rho^1,u^1,(\tau^1_R), \tau^1)$ and $(\rho^2,u^2,(\tau^2_R), \tau^2)$ being two different solutions of \eqref{main system} in the sense of Definition \ref{Def2},  with the same initial data $(\rho_0,u_0)$ then $(r^{i},u^{i})$ with $r^i. \; i=1,2$ defined by \eqref{new variables} are two solutions of \eqref{trunc local eq r}-\eqref{trunc local eq u} up to the stopping time $\tau^1_R \land \tau^2_R.$ Therefore, the pathwise uniqueness result given by Theorem \ref{well posedness trunc system} implies that 
\begin{equation*}
\mathbb{P} \big( [\rho^1,u^1](t \land \tau^1_R \land \tau^2_R )= [\rho^2,u^2](t \land \tau^1_R \land \tau^2_R ), \; \forall t \in [0,T] \big)=1.
\end{equation*}
and our claim follows by sending $R \rightarrow \infty,$ dominated convergence Theorem and the maximality of $\tau^1$ and $\tau^2.$
\\
\\
To prove the existence result of Theorem \ref{Main Theorem local} we consider a strong solution of the approximating system \eqref{trunc local eq r}-\eqref{trunc local eq u} which is denoted by $(r_R,u_R)$ and we introduce the following stopping time 
\begin{equation*}
\tau_R= \inf \{ t \in [0,T] | \; \| r_R (t) \|_{W^{2,\infty}} + \| u_R(t) \|_{W^{2,\infty}} \ge R \}, \quad \; \inf \emptyset=T,
\end{equation*}
from which we aim to recover a strong pathwise solution of the original system in the sense of Definition \ref{Def1}.
Specifically since $r_R$ and $u_R$ have continuous trajectories in $H^{s+1}$ and $H^{s}$ respectively,  they are embedded into $W^{2,\infty},$ then $\tau_R$ is a well-defined $a.s.$ strictly positive stopping time. Therefore $[r_R,u_R]$ generates the local strong pathwise solution $(\rho,u, \tau_R)$ of system \eqref{main system} with 
\begin{equation*}
\rho:= \begin{cases} 
\big( \frac{\beta+1}{2} \big)^{\frac{2}{\beta+1}} r^{\frac{2}{\beta+1}} _R, \quad \beta \neq -1, \\ 
e^{r_R},  \quad \quad \quad \quad \quad  \;  \; \beta=-1.
\end{cases}
\end{equation*}
and $u=u_R,$ with initial condition $(\rho_0,u_0).$
\\
\\
The construction of a maximal stopping time $\tau$ in the sense of Definition \ref{Def2} is quite standard in the analysis of stochastic partial differential equations.  Specifically, the same strategy has been already applied in \cite{Feir} and \cite{D.P.S.} for compressible fluid models.  Nevertheless, we highlight some elements. The set of all possible a.s. strictly positive stopping times $\mathsf{F}$ corresponding to the solution starting from the initial datum $[\rho_0,u_0]$ is non-empty and it is closed with respect to finite maximum and finite minimum operations.  Then, given $\sigma \in \mathsf{F}$ the maximal stopping time $\tau$ is defined by $\tau= \text{ess} \sup_{\sigma \in \mathsf{F}} \sigma$ and the proof follows from uniqueness of solutions and standard arguments. 
\section{Proof of the main result}\label{Sec5}
\noindent
Finally, we prove the main result of the paper Theorem \ref{Main Theorem global}. We provide an extension argument which is based on the a priori estimates derived in Section \ref{Sec3} and the properties of the unique local maximal strong pathwise solutions constructed in Section \ref{Sec4}.  For completeness, we point the reader to \cite{Deb} and \cite{Glatt-Holtz} for similar stopping time arguments used in the proof of global existence of solutions for  incompressible fluid models.
\\
\\
Assume the viscosity and capillarity exponents $\alpha$ and $\beta$ are in the range given by conditions \eqref{SCC} and \eqref{NV}. Let  $(\rho,u,(\tau_n)_{n \in \mathbb{N}},\tau)$ being the unique maximal strong pathwise solution of \eqref{main system} in the sense of Definition \ref{Def2} with initial data in the regularity class \eqref{C.I. Momenta strong}. We claim that $\tau= \infty$ a.s.
To this purpose, we recall that 
\begin{equation*}
\tau_n= \inf \{ t \in [0,\infty) | \; \| r(t) \|_{W^{2,\infty}}+ \| u(t) \|_{W^{2,\infty}} \ge n \}, \quad \inf \emptyset= \infty,
\end{equation*}
which is announcing $\tau,$ meaning that $\lim_{n \rightarrow \infty} \tau_n=\tau $ a.s. 
First we make use of the following standard decomposition 
\begin{equation*}
\{ \tau < \infty \}= \bigcup_{T=1}^{\infty}\{ \tau \le T \}= \bigcup_{T=1}^{\infty} \bigcap_{n=1}^{\infty} \{ \tau_n \le T \}.
\end{equation*}
Then by monotonicity of $\tau_n$ we have
\begin{equation*}
\mathbb{P} \big( \bigcap_{n=1}^{\infty} \{ \tau_n \le T \} \big)= \lim_{N \rightarrow \infty}\mathbb{P} \big( \bigcap_{n=1}^{N} \{ \tau_n \le T \} \big)= \lim_{N \rightarrow \infty} \mathbb{P} \big( \tau_N \le T) \big).
\end{equation*}
Hence in order to prove our claim it is enough to show that for any fixed $T< \infty$
\begin{equation}\label{lim N}
\lim_{N \rightarrow \infty} \mathbb{P}( \tau_N \le T)= 0.
\end{equation}
To this end we make use of the auxiliary stopping time $\gamma_M,$ see \eqref{min gamma M}, used in the derivation of the high-order estimates given in Section \ref{Sec3}. More precisely, the main idea is to control the Sobolev norms appearing in the announcing sequence of stopping times $\tau_N$ by virtue of Proposition \ref{Prop global s+1 s}. Our claim then follows by sending $M \rightarrow \infty.$ 
To be more explicit, we use 
\begin{equation}
\bar{\gamma}_M:= \gamma_M \land 2T
\end{equation}
and we do not relabel the above stopping time in order to avoid heavy notations. \\
\\
Then we use the following inequality
\begin{equation}\label{prob int}
\mathbb{P}\big( \tau_N \le T\big) \le \mathbb{P}\big( \{ \tau_N \le T \} \cap \{ \gamma_M > T \}\big)+ \mathbb{P}\big( \gamma_M \le T\big).
\end{equation}
Therefore by definition of $\tau_N,$ Sobolev embedding and Markov inequality we have 
\begin{equation}\label{ineq final global}
\begin{split}
& \mathbb{P} \big( \{ \tau_N \le T \} \cap \{ \gamma_M > T \} \big) \le \mathbb{P} \big( \big\{ \sup_{t \in [0, \tau_N \land T]} ( \| A \|^2_{H^3}+ \| u \|^2_{H^3} ) \ge N \big\} \cap \{ \gamma_M > T \} \big) \\ & \le \mathbb{P} \big( \big\{ \sup_{t \in [0, \tau_N \land \gamma_M]} ( \| A \|^2_{H^3}+ \| u \|^2_{H^3} ) \ge N \big\} \big) \le \frac{1}{N} \mathbb{E} \bigg[ \sup_{t \in [0, \tau_N \land \gamma_M]} \big( \| A \|^2_{H^3}+ \| u \|^2_{H^3} \big) \bigg].
\end{split}
\end{equation}
Thus by virtue of Proposition \ref{Prop global s+1 s} we send $N \rightarrow \infty$ in \eqref{ineq final global} and use \eqref{prob int}  to deduce that for any fixed $M>0$
\begin{equation}\label{stima gamma M}
\lim_{N \rightarrow \infty} \mathbb{P}( \tau_N < T) \le \mathbb{P}(\gamma_M \le T).
\end{equation}
To conclude we recall \eqref{gamma1}, \eqref{gamma2}, \eqref{gamma3}, \eqref{min gamma M} and we take the limit $M \rightarrow \infty$ in \eqref{stima gamma M}. The proof is therefore concluded.
\appendix
\section{}\label{Deterministic tools}
\noindent
In this short appendix we present some functional inequalities used in the derivation of the BD entropy result Proposition \ref{BD prop} and other technical lemmas. The following lemma can be found in \cite{Alazard}, Proposition $2.8$, see also \cite{Germain LeFloch}, Theorem 2.1 for a similar result. 
\begin{lemma}\label{Functional ineq }
For any $d \ge 1 $ and any positive function $f$ in $H^2(\mathbb{T}^d),$
\begin{equation*}
\int_{\mathbb{T}^d} | \nabla f^\frac{1}{2}|^4  dx \le \dfrac{9}{16} \int_{\mathbb{T}^d} (\Delta f)^2 dx.
\end{equation*}
\end{lemma}
\noindent
Note that the constant $\frac{9}{16}$ is optimal in the sense that if the above inequality holds for a constant $c >\frac{9}{16}$ then $f$ is constant.
\noindent
The characterization of the strong coercivity condition \eqref{SCC} is given in the proposition below. Although the proof follows the lines of  Proposition 2.3 in \cite{Ant}, we rewrite it here to account for a modified definition of the effective velocity $Q,$ which has been tailored to include the case of constant viscosity.
\begin{proposition}\label{char scc}
Let $D_{\alpha,\beta}[\rho]$ be defined as in \eqref{entropy D cap}. Then $D_{\alpha,\beta}[\rho] \ge 0$ if and only if \eqref{SCC} holds.
Furthermore, assume that 
\begin{equation}\label{not sharp scc}
2\alpha-4 < \beta < 2\alpha-1,
\end{equation}
then for $\theta:= \frac{\alpha+\beta+1}{2}$ there exist two positive constants $C_1, \; C_2$ depending on $\alpha$ and $\beta$ such that 
\begin{itemize}
\item [1)] For $\theta \neq 0$
\begin{equation}\label{D not sharp theta}
\int_{\mathbb{T}} | \partial_{xx} \rho^{\theta}|^2 dx + \int_{\mathbb{T}} | \partial_x \rho^{\frac{\theta}{2}} |^4 dx \le C_1 D_{\alpha,\beta}[\rho],
\end{equation}
\item [2)] For $\theta=0$
\begin{equation}\label{D not sharp theta=0}
\int_{\mathbb{T}} | \partial_{xx} \log \rho|^2 dx + \int_{\mathbb{T}} | \partial_x \log \rho |^4 dx \le C_2 D_{\alpha,\beta}[\rho].
\end{equation}
\end{itemize}
\end{proposition}
\begin{proof}
By definition of $D_{\alpha,\beta}$ and $c(\alpha,\beta)$ we have that 
\begin{equation}
\theta^2 \mathcal{D}_{\alpha,\beta}[\rho]= \int_{\mathbb{T}} \big| \partial_{xx} \rho^{\frac{\alpha+ \beta+1}{2}} \big|^2 +  16 \bigg[ \dfrac{(\alpha-\beta-1)(1-\alpha)}{(\alpha+\beta+1)^2}-\dfrac{\beta}{3(\alpha+\beta+1)} \bigg] \big| \partial_x \rho^{\frac{\alpha+\beta+1}{4}} \big|^4 dx,
\end{equation}
On the other hand, by using \eqref{Functional ineq } with $d=1$ and $f=\rho^{\theta}$ we have that 
\begin{equation*}
\theta^2 D_{\alpha,\beta}[\rho] \ge 16 \bigg[ \dfrac{(\alpha-\beta-1)(1-\alpha)}{(\alpha+\beta+1)^2}-\dfrac{\beta}{3(\alpha+\beta+1)} +\frac{1}{9} \bigg] \int_{\mathbb{T}} \big| \partial_x \rho^{\frac{\theta}{2}} \big|^4 dx
\end{equation*}
and we observe that condition \eqref{SCC} is equivalent to 
\begin{equation}\label{optimal inequ}
\dfrac{(\alpha-\beta-1)(1-\alpha)}{(\alpha+\beta+1)^2}-\dfrac{\beta}{3(\alpha+\beta+1)} +\frac{1}{9} \ge 0,
\end{equation}
therefore \eqref{SCC} holds and it is also necessary since the constant $\frac{9}{16}$ is optimal. Clearly, by assuming \eqref{not sharp scc}, then \eqref{D not sharp theta} holds since \eqref{optimal inequ} is a sharp inequality. The case $\theta=0$ follows the exact same lines of argument with respect to the previous case.
\end{proof}
\noindent
\textbf{Declarations}
\\
\\
\textbf{Data Availability.} The authors declare that data sharing is not applicable to this article as no datasets were generated or analysed.
\\
\\
\textbf{Conflict of interest.} The authors declare that they have no conflict of interest.


\begin{thebibliography}{99}
\bibitem{Alazard} T. Alazard and D. Bresch,
\newblock \textit{Functional inequalities and strong Lyapunov functionals for free surface flows in fluid dynamics},
\newblock Interfaces Free Bound.,(2023).

\bibitem{Ant} P. Antonelli,  D. Bresch and S. Spirito,
\newblock \textit{Global weak solutions of the Navier-Stokes-Korteweg equations in one dimension},
\newblock (2025), Preprint available at arXiv:2502.17147.

\bibitem{Ant Cac}  P. Antonelli and Y. Cacchiò
\newblock \textit{The zero capillarity limit for the Euler-Korteweg system with no-flux boundary conditions}
\newblock (2025), Preprint available at arxiv:2510.27682.

\bibitem{Ant0} P. Antonelli, G. Cianfarani Carnevale, C. Lattanzio, and S. Spirito,
\newblock \textit{Relaxation limit from the quantum Navier-Stokes equations to the quantum drift-diffusion equation},
\newblock J. Nonlinear Sci. 31 (2021), no. 5, Paper No. 71, 32 pp. 1.

\bibitem{Ant1} P. Antonelli and P. Marcati,
\newblock \textit{On the finite energy weak solutions to a system in quantum fluid dynamics},
\newblock Comm.Math. Phys., 287, (2009),  no.2,  657-686.

\bibitem{Ant2} P. Antonelli and P. Marcati,
\newblock \textit{The Quantum Hydrodynamics system in two space dimensions},
\newblock Arch. Ration. Mech. Anal., 203, (2012), no.2, 499-527.

\bibitem{Ant3} P. Antonelli, P. Marcati, and H. Zheng, 
\newblock \textit{An intrinsically hydrodynamic approach to multidimensional {QHD} systems},
\newblock Arch. Ration. Mech. Anal., 247, (2023), Paper No. 24, 58.

\bibitem{Ant4} P. Antonelli, P. Marcati, and H. Zheng, 
\newblock \textit{Genuine hydrodynamic analysis to the 1-{D} {QHD} system: existence, dispersion and stability},
\newblock Comm. Math. Phys., 383, (2021),  no.3, 2113--2161.

\bibitem{Ant5} P. Antonelli, L. E. Hientzsch, and P. Marcati, 
\newblock \textit{On the low Mach number limit for quantum Navier-Stokes equations},
\newblock SIAM J. Math. Anal., 52, (2020), no.6,  6105-6139.

\bibitem{Ant6} P. Antonelli, L.E.  Hientzsch, and S. Spirito
\newblock \textit{Global existence of finite energy weak solutions to the quantum Navier-Stokes equations with non-trivial far-field behavior},
\newblock J. Differential Equations, 290, (2021), 147-177.

\bibitem{Spirito} P. Antonelli and S. Spirito,
\newblock \textit{Global existence of weak solutions to the Navier-Stokes-Korteweg equations},
\newblock Ann.  Inst.  H.  Poincarè C. Anal. Non Linèaire 39 (2022),  no.1, pp 171-200.

\bibitem{Spirito2} P. Antonelli and S. Spirito,
\newblock \textit{On the compactness of finite energy weak solutions to the quantum Navier-Stokes equations},
\newblock J.  Hyperbolic Differ.  Equ.,  15, (2018),  no.1,  pp. 133-147.

\bibitem{Aud} C. Audiard and B. Haspot,
\newblock \textit{Global well-posedness of the Euler-Korteweg system for small irrotational data},
\newblock Comm. Math. Phys., 351, (2017), no. 1,  201-247.


\bibitem{Benz} S. Benzoni-Gavage, R. Danchin, and S. Descombes,
\newblock \textit{On the well-posedness for the Euler-Korteweg model in several space dimensions},
\newblock Indiana Univ. Math. J., 56, (2007),  no.4, 1499-1579.

\bibitem{Benz2} S. Benzoni-Gavage, R. Danchin and S. Descombes, \newblock \textit{Well-posedness of one-dimensional Korteweg models}, 
\newblock Electron. J. Differential Equations {2006}, No. 59, 35 pp..


\bibitem{Breit Feir Hof 1} D. Breit, E. Feireisl and M. Hofmanov\'{a}
\newblock \textit{Incompressible limit for compressible fluids with stochastic forcing},
\newblock Arch. Rational Mech. Anal. 222, 895–926, 2016.

\bibitem{Breit Feir Hof 2} D. Breit, E. Feireisl, and M. Hofmanov\'{a}
\newblock \textit{Compressible fluids driven by stochatic forcing: The relative energy
inequality and applications},
\newblock Commun. Math. Phys. 350, 443–473, 2017.

\bibitem{Breit Feir Hof 3} D. Breit, E. Feireisl and M. Hofmanov\'{a},
\newblock \textit{Local strong solutions to the stochastic compressible Navier-Stokes
system},
\newblock Comm. Partial Differential Equations 43 (2018), no.~2, 313--345.

\bibitem{Feir} D. Breit,  E. Feireisl,  and M. Hofmanov\'{a}, 
\newblock Stochastically Forced Compressible Fluid Flows,
\newblock De Gruyter Series in Applied and Numerical Mathematics. De Gruyter (2018).

\bibitem{Breit Feir Hof 4} D. Breit, E. Feireisl and M. Hofmanov\'{a},
\newblock \textit{On solvability and ill-posedness of the compressible Euler system subject to stochastic forces},
\newblock Anal. PDE 13 (2020), no.~2, 371--402.

\bibitem{Breit Feir Hof 5} D. Breit, E. Feireisl and M. Hofmanov\'{a}, 
\newblock \textit{On the long-time behavior of compressible fluid flows excited by random forcing}, 
\newblock Ann. Inst. H. Poincar\'e{} C Anal. Non Lin\'eaire {41} (2024), no.~4, 961--993.

\bibitem{Breit Feir Hof Much} D. Breit,  E. Feireisl, M. Hofmanov\'{a} and P. B. Mucha,
\newblock \textit{Compressible fluids excited by space-dependent transport noise},
\newblock Preprint (2025), available at arXiv.2504.10256.

\bibitem{Breit Feir Hof Zat} D. Breit, E. Feireisl, M. Hofmanov\'{a} and E. Zatorska, \newblock \textit{Compressible Navier-Stokes system with transport noise},
\newblock SIAM J. Math. Anal. {54} (2022), no.4, 4465-4494.

\bibitem{Breit Hofmanova} D. Breit and M. Hofmanov\'{a}, 
\newblock \textit{Stochastic Navier-Stokes equations for compressible fluids},
\newblock Indiana Univ. Math. J. 65, 1183–1250, 2016.

\bibitem{BD} D. Bresch and B. Desjardins,
\newblock \textit{Sur un modèòe de Saint-Venant visqueux et sa limite quasi géostrophique, [ On viscous shallow-water equations (Saint-Venant model) and the quasi-geostrophic limit]},
\newblock C.R. Math. Acad. Sci. Paris, 335, (2002),  no.12, 1079-1084.

\bibitem{Bresch2} D. Bresch and B. Desjardins, 
\newblock \textit{Existence of global weak solutions for a 2D viscous shallow water equations and convergence to the quasi-geostrophic model}, 
\newblock Comm. Math. Phys. {238} (2003), no.~1-2, 211--223.

\bibitem{Bresch0} D. Bresch, M. Gisclon and I. Lacroix-Violet, 
\newblock \textit{On Navier-Stokes-Korteweg and Euler-Korteweg systems: application to quantum fluids models},
\newblock Arch. Ration. Mech. Anal. {\bf 233} (2019), no.~3, 975--1025

\bibitem{Zatorska} Z.Brzeźniak, G. Dhariwal and E. Zatorska,
\newblock \textit{Sequential stability of weak martingale solutions to stochastic compressible Navier-Stokes equations with viscosity vanishing on vacuum},
\newblock Journal of Differential Equations, Volume 416, Part 2,2025,Pages 1285-1346

\bibitem{Burtea1} C. Burtea and B. Haspot,
\newblock \textit{Existence of global strong solution for the Navier–Stokes–Korteweg system in one dimension for strongly degenerate viscosity coefficients},
\newblock Pure Appl. Anal.,  4, (2022), no. 3, 449-485.

\bibitem{Burtea2} C. Burtea and B. Haspot,
\newblock \textit{Vanishing capillarity limit of the {N}avier-{S}tokes-{K}orteweg system in one dimension with degenerate viscosity coefficient and discontinuous initial density},
\newblock SIAM J. Math. Anal., 54, (2022),  no.2, 1428--1469.

\bibitem{Donatelli} M. Caggio and D. Donatelli, 
\newblock \textit{High Mach number limit for Korteweg fluids with density dependent viscosity},
\newblock J. Differential Equations, 277,  (2021),  1-37.

\bibitem{Cag-Don-Hien} M. Caggio, D. Donatelli and L.E.  Hientzsch
\newblock \textit{Inviscid incompressible limit for capillary fluids with density dependent viscosity},
\newblock (2025), Preprint available at arxiv:2507.00621.

\bibitem{Charve} F.  Charve and B. Haspot,
\newblock \textit{Existence of global strong solution and vanishing capillarity-viscosity limit in one dimension for the Korteweg system},
\newblock SIAM J.  Math. Anal.  45, (2013), no.2,  469-494.

\bibitem{Chen} Z. Chen, X. Chai, B. Dong,  and H. Zhao,
\newblock \textit{Global classical solutions to the one-dimensional compressible fluid models of Korteweg type for large initial data},
\newblock J. Differential Equations, 259, (2015), no.8, 4376-4411.

\bibitem{Cho 2004} Y. Cho,  H. J.  Choe, and H.  Kim,
\newblock \textit{Unique solvability of the initial boundary value problems for compressible viscous fluids},
\newblock J.  Math. Pures Appl., 83, (2004), no.2, 243-275.

\bibitem{Const} P. Constantin, T.D. Drivas,  H.Q. Nguyen,  and F.  Pasqualotto,
\newblock \textit{Compressible fluids and active potentials},
\newblock Ann. Inst. H. Poincaré C. Anal. Non Linéaire,  37,  (2020), no.1, pp. 145-180.

\bibitem{Const2} P. Constantin, T.D. Drivas,  and R. Shvydkoy,
\newblock \textit{Entropy hierarchies for equations of compressible fluids and self-organized dynamics},
\newblock SIAM J. Math. Anal., 52, (2020), no. 3, 3073-3092.

\bibitem{Coti} M. Coti-Zelati,  N. Glatt-Holtz,  and K. Trivisa,
\newblock \textit{Invariant measure for the stochastic compressible Navier-Stokes equations},
\newblock Appl. Math. Optim., 83, (2021), no. 3, 1487-1522.

\bibitem{Da Prato} G. Da Prato and J. Zabczyk,
\newblock Stochastic Equations in Infinite Dimensions,
\newblock volume 44 of \textit{Encyclopedia of Mathematics and its Applications},  Cambridge University Press, Cambridge, 1992.

\bibitem{Deb} A. Debussche, N. Glatt–Holtz, R. Temam, and M. Ziane,
\newblock \textit{Global Existence and Regularity for the 3D Stochastic Primitive Equations of the Ocean and Atmosphere with Multiplicative White Noise},
\newblock Nonlinearity, 25, (2012), no.7, 2093–2118.

\bibitem{Denzler} J. Denzler and R.J. McCann,
\newblock \textit{Nonlinear diffusion form a delocalized source: affine self-similariy, time reversal and nonradial focusing geometries},
\newblock Ann. Inst. H. Poincaré C Anal. Non Linéaire 25 (2008), no. 5, 865-888.

\bibitem{Feir. Don.} D. Donatelli, E. Feireisl,  and P. Marcati,
\newblock \textit{Well/ill Posedness for the Euler-Korteweg-Poisson System and Related Problems},
\newblock Comm.  Partial Differential Equations, 40, (2015),  no.7,  pp. 1314-1335.

\bibitem{Don-Pes} D. Donatelli and L. Pescatore,
\newblock \textit{Blow up criteria for a fluid dynamical model arising in astrophysics},
\newblock  J. Hyperbolic Differ.  Equ., 20, (2023), no.3,  pp. 629-668.

\bibitem{D.P.S.} D. Donatelli, L. Pescatore and S. Spirito, 
\newblock \textit{Global regularity for the one-dimensional stochastic Quantum-Navier-Stokes equations}
\newblock (2024), Preprint, available at arXiv:2401.10064.

\bibitem{D.P.S.2} D. Donatelli, L. Pescatore and S. Spirito, 
\newblock \textit{Weak martingale solutions to the stochastic 1D Quantum-Navier-Stokes equations}, 
\newblock (2024), Preprint, available at arXiv:2412.10875.

\bibitem{Dunn and Serrin} J.E. Dunn, J. Serrin 
\newblock \textit{On the thermomechanics of interstitial working}, 
\newblock Arch. Rational Mech. Anal. 88, 95–133 (1985). 

\bibitem{Fanelli1} F. Fanelli, 
\newblock \textit{Highly rotating viscous compressible fluids in presence of capillarity effects}, 
\newblock Math. Ann. {366} (2016), no.~3-4, 981--1033.

\bibitem{Fanelli2} F. Fanelli, 
\newblock \textit{A singular limit problem for rotating capillary fluids with variable rotation axis}, 
\newblock J. Math. Fluid Mech. {18} (2016), no.~4, 625--658.

\bibitem{Germain LeFloch} P. Germain and P.  LeFloch,
\newblock \textit{Finite Energy Method for Compressible Fluids: The Navier-Stokes-Korteweg Model},
\newblock Comm. Pure Appl. Math., 69, (2015), no.1, 3-61.

\bibitem{Giess} J. Giesselmann, C. Lattanzio and A.E. Tzavaras, 
\newblock \textit{Relative energy for the Korteweg theory and related Hamiltonian flows in gas dynamics}, 
\newblock Arch. Ration. Mech. Anal. {\bf 223} (2017), no.3, 1427--1484.

\bibitem{Glatt-Holtz} N. Glatt-Holtz and M. Ziane.
\newblock \textit{Strong pathwise solutions of the stochastic Navier-Stokes system.}
\newblock Adv.  Differential Equations,  14, (2009), no.5-6: 567-600.

\bibitem{Gorban} A.N. Gorban and I.V. Karlin, 
\newblock \textit{Hilbert's 6th problem: exact and approximate hydrodynamic manifold for kinetic equations},
\newblock Bull. Amer. Math. Soc. (N.S.) 51 (2014), no.2, 187-246.

\bibitem{Haspot} B. Haspot,
\newblock \textit{Existence of Global Weak Solution for Compressible Fluid Models of Korteweg Type},
\newblock J. Math. Fluid Mech., 13, (2011),  no.2, 223-249.

\bibitem{Hatt} H. Hattori and D. Li,
\newblock \textit{Solutions for two dimensional systems for materials of Korteweg type},
\newblock SIAM J.Math.Anal., 25, (1994), no.1,  85-98.

\bibitem{Jung qns} A. J\"{u}ngel,  
\newblock \textit{Dissipative quantum fluid models},
\newblock Riv.  Mat. Univ.  Parma 3 (2012),  no.2, 217--290.

\bibitem{Jung qns 2} A. J\"{u}ngel, 
\newblock \textit{Global weak solutions to compressible Navier–Stokes equations for quantum fluids}, 
\newblock SIAM J. Math. Anal. 42 (2010), 1025--1045.

\bibitem{Jungel2} A. J\"ungel, C.-K. Lin and K.-C. Wu, 
\newblock \textit{An asymptotic limit of a Navier-Stokes system with capillary effects},
\newblock Comm. Math. Phys. {329} (2014), no.~2, 725--744.

\bibitem{Slemrod Kim} Y. J. Kim, M. G. Lee and M. Slemrod, 
\newblock \textit{Thermal creep of a rarefied gas on the basis of non-linear Korteweg-theory},
\newblock Arch. Ration. Mech. Anal.  215 (2015), no.~2, 353--379.

\bibitem{Korteweg} D.J. Korteweg,
\newblock \textit{Sur la forme que prennent les équations du mouvements des fluides si l'on tient copmpte des forces capillaires causées par des variations de densité considérables mais connues et sur la théorie de la capillarité dans l'hypothése d'une variation continue de la densité},
\newblock Archives Néerlandaises des ciences exactes et naturelles 6 Ser. II, 1-24,1901.

\bibitem{Lacroix} I. Lacroix-Violet and A. Vasseur,
\newblock \textit{Global weak solutions to the compressible quantum Navier-Stokes equation and its semi-classical limit},
\newblock J. Math. Pures Appl., 114 (2017), 191--210.

\bibitem{Landau} L. Landau and E. Lifschitz,
\newblock \textit{Quantum Mechanics: Non-relativistic Theory.}
\newblock Pergamon Press,  New York, 1977.

\bibitem{Mellet} A. Mellet and A. Vasseur,
\newblock \textit{Existence and uniqueness of global strong solutions for one-dimensional compressible Navier-Stokes equations,}
\newblock SIAM J. Math. Anal., 39, (2008), no. 4,  pp. 1344–1365.

\bibitem{Pescatore} L. Pescatore,
\newblock \textit{The stochastic compressible Navier-Stokes-Korteweg system: The one-dimensional case},
\newblock PhD Thesis, University of L'Aquila, (2025).

\bibitem{Vasseur Yu 0} A.F.  Vasseur and C. Yu, 
\newblock \textit{Existence of global weak solutions for 3D degenerate compressible Navier-Stokes equations}, 
\newblock Invent. Math.  206 (2016), no.~3, 935--974.

\bibitem{Vasseur Yu 1} A.F.  Vasseur and C. Yu, 
\newblock \textit{Global weak solutions to the compressible quantum Navier-Stokes equations with damping}, 
\newblock SIAM J. Math. Anal., 48 (2016), 1489-1511.

\end{thebibliography}
\end{document}